\documentclass[11pt,reqno]{amsart}

\usepackage{fullpage}

\usepackage{times}
\usepackage[T1]{fontenc}
\usepackage{mathrsfs}
\usepackage{latexsym}
\usepackage[dvips]{graphics}
\usepackage[dvips]{graphicx}
\usepackage{epsfig}
\usepackage{amsmath,amsfonts,amsthm,amssymb,amscd}
\input amssym.def
\input amssym.tex
\usepackage{hyperref}
\usepackage{url}
\usepackage{color}
\newcommand{\burl}[1]{\textcolor{blue}{\url{#1}}}

\usepackage{verbatim}

\newcommand{\gl}{\lambda}

\newcommand{\ncr}[2]{\left({#1 \atop #2}\right)}

\newcommand{\threecase}[7]{#1 \begin{cases} #2 & \text{\rm #3}\\ #4
&\text{\rm #5}\\ #6 & \text{\rm #7} \end{cases}  }

\newtheorem{thm}{Theorem}[section]

\newtheorem{cor}[thm]{Corollary}
\newtheorem{lem}[thm]{Lemma}
\newtheorem{pro}[thm]{Proposition}
\newtheorem{defi}[thm]{Definition}

\numberwithin{equation}{section}
\newtheorem{rek}[thm]{Remark}

\newcommand\be{\begin{equation}}
\newcommand\ee{\end{equation}}
\newcommand\bea{\begin{eqnarray}}
\newcommand\eea{\end{eqnarray}}
\newcommand\bi{\begin{itemize}}
\newcommand\ei{\end{itemize}}
\newcommand\ben{\begin{enumerate}}
\newcommand\een{\end{enumerate}}

\renewcommand{\geq}{\geqslant}
\renewcommand{\leq}{\leqslant}
\renewcommand{\mod}[1]{
	{\ifmmode\text{\rm\ (mod~$#1$)}\else
		\discretionary{}{}{\hbox{ }}\rm(mod~$#1$)\fi}
}

\newcommand{\R}{\ensuremath{\mathbb{R}}}

\newcommand{\N}{\ensuremath{\mathbb{N}}}

\newcommand{\E}{\ensuremath{\mathbb{E}}}

\newsymbol\dnd 232D
\newcommand{\Ln}{L_{i, i +k}(n)}
\newcommand{\Rn}{R_{i, i +k}(n)}
\newcommand{\Xn}{X_{i, i +k}(n)}
\newcommand{\Xnn}{X_{i, i +1}(n)}

\newcommand{\km}{\frac{1}{k(m)-1}}

\newcommand{\Cn}{C_{{\rm Lek}}n + d + o(1)}

\newcommand{\hnumnt}{\widehat{\nu_{m;n}}(t)}
\newcommand{\hnut}{\widehat{\nu}(t)}

\newtheorem{prop}[thm]{Proposition}
\newtheorem{defn}[thm]{Definition}

\newcommand{\comments}[1]{}

\newcommand{\calR}{\mathcal{R}}
\newcommand{\calM}{\mathcal{M}}
\newcommand{\calG}{\mathcal{G}}

\newcommand{\bbN}{{\mathbb{N}}}
\newcommand{\bbR}{{\mathbb{R}}}

\newcommand{\bbC}{{\mathbb{C}}}

\newcommand\beq{\begin{equation}}
\newcommand\eeq{\end{equation}}

\begin{document}

\setcounter{page}{1}

\title[Gaps between Summands in Generalized Zeckendorf Decompositions]{The Distribution of Gaps between Summands in Generalized Zeckendorf Decompositions}

\author{Amanda Bower}
\email{\textcolor{blue}{\href{mailto:amandarg@umd.umich.edu}{amandarg@umd.umich.edu}}}
\address{Department of Mathematics and Statistics, University of Michigan-Dearborn, Dearborn, MI 48128}

\author{Rachel Insoft}
\email{\textcolor{blue}{\href{mailto:rinsoft@wellesley.edu}{rinsoft@wellesley.edu}}}
\address{Department of Mathematics, Wellesley College,  Wellesley, MA 02481}

\author{Shiyu Li}
\email{\textcolor{blue}{\href{mailto:jjl2357@berkeley.edu}{jjl2357@berkeley.edu}}}
\address{Department of Mathematics, University of California, Berkeley, Berkeley, CA 94720}

\author{Steven J. Miller}\email{\textcolor{blue}{\href{mailto:sjm1@williams.edu}{sjm1@williams.edu}},  \textcolor{blue}{\href{Steven.Miller.MC.96@aya.yale.edu}{Steven.Miller.MC.96@aya.yale.edu}}}
\address{Department of Mathematics and Statistics, Williams College, Williamstown, MA 01267}

\author{Philip Tosteson}
\email{\textcolor{blue}{\href{mailto:Philip.D.Tosteson@williams.edu}{Philip.D.Tosteson@williams.edu}}}
\address{Department of Mathematics and Statistics, Williams College, Williamstown, MA 01267}

\date{\today}



\thanks{The fourth named author was partially supported by NSF grants DMS0970067 and DMS1265673, and the remaining authors were partially supported by NSF Grant DMS0850577. It is a pleasure to thank our colleagues from the Williams College 2010, 2011, 2012 and 2013 SMALL REU program for many helpful conversations, and Philippe Demontigny and Cameron Miller for discussions on generalizations.}

\begin{abstract} Zeckendorf proved that any integer can be decomposed uniquely as a sum of non-adjacent Fibonacci numbers, $F_n$. Using continued fractions, Lekkerkerker proved the average number of summands of an $m \in [F_n, F_{n+1})$ is essentially $n/(\varphi^2 +1)$, with $\varphi$ the golden ratio. Miller-Wang generalized this by adopting a combinatorial perspective, proving that for any positive linear recurrence the number of summands in decompositions for integers in $[G_n, G_{n+1})$ converges to a Gaussian distribution. We prove the probability of a gap larger than the recurrence length converges to decaying geometrically, and that the distribution of the smaller gaps depends in a computable way on the coefficients of the recurrence. These results hold both for the average over all $m \in [G_n, G_{n+1})$, as well as holding almost surely for the gap measure associated to individual $m$. The techniques can also be used to determine the distribution of the longest gap between summands, which we prove is similar to the distribution of the longest gap between heads in tosses of a biased coin. It is a double exponential strongly concentrated about the mean, and is on the order of $\log n$ with computable constants depending on the recurrence.
\end{abstract}

\maketitle

\tableofcontents

\bigskip
\bigskip

\section{Introduction}

\subsection{Background}

In this paper we explore the distribution between summands in generalized Zeckendorf decompositions. Before stating our results, we first quickly motivate the problem and summarize previous work.

There are many ways to decompose integers. The most familiar are of course binary and decimal expansions, but there are many others. For example, conjecturally every even integer at least 4 can be written as the sum of two primes. While this has the enormous advantage of giving highly sparse representations (if we let 1 represent a prime that is chosen and 0 one that is not, most primes are not chosen), these decompositions have the undesirable property that a given element typically does not have a unique decomposition. We desire a decomposition between these extremes with the following properties: (1) existence (every positive integer has a decomposition), (2) uniqueness (there is only one decomposition for each number), and (3) sparseness (many of the possible summands are not used). The latter property suggests that such decompositions can have applications in computer science, where storage costs are a major issue.

Fortunately, there are many examples satisfying these three properties. A famous one is the Zeckendorf decomposition. Zeckendorf  \cite{Ze} proved that every positive integer can be written uniquely as a sum of non-adjacent Fibonacci numbers. Here the Fibonacci numbers are given by $F_1 = 1, F_2 = 2$ and $F_{n+2} = F_{n+1} + F_n$; it is imperative that we do not start the Fibonacci sequence with 0 and 1 (if we did we lose uniqueness). The standard proof is by a greedy algorithm. Given an integer $m$ let $F_j$ be the largest Fibonacci number at most $m$. Let $F_\ell$ be the largest Fibonacci number less than $m-F_j$. If $\ell = j-1$ then $F_j + F_{j-1} \le m$, which implies $F_{j+1} \le m$. This contradicts the maximality of $F_j$, and thus $\ell \le j-2$; by induction we are done. This proof illustrates the naturalness of the non-adjacency condition.

We can ask many questions about the Zeckendorf decomposition. The most basic concerns the average number of summands needed; clearly the answer is less than 50\% as we cannot have two adjacent summands. Lekkerkerker \cite{Lek} proved that for $m \in [F_n, F_{n+1})$, as $n\to\infty$ the average number of summands needed is $n/(\varphi^2 + 1)$, with $\varphi = \frac{1+\sqrt{5}}2$ the golden mean. More generally, we may replace the Fibonacci numbers with other sequences and ask whether or not a decomposition exists with our three desired properties. The following theorem gives a large class of recurrence relations where such a decomposition exists, and gives the analogue of non-adjacency (essentially we cannot use the recurrence relation to reduce our decomposition). See for example \cite{MW1,MW2} for a proof and \cite{BCCSW,Day,GT,Ha,Ho,Ke,Len} for some of the history and results along these lines.

\begin{thm}[Generalized Zeckendorf Decomposition and Generalized Lekkerkerker's Theorem]\label{thm:legaldecomp} Consider a \textbf{positive linear recurrence} \be G_{n+1} \ = \ c_1 G_{n} + \cdots + c_L G_{n+1-L} \ee with non-negative integer coefficients $c_i$ with $c_1, c_L > 0$, and initial conditions $G_1 = 1$ and for $1 \le n \le L$ \be G_{n+1} \ =\ c_1 G_n + c_2 G_{n-1} + \cdots + c_n G_{1}+1. \ee  For each positive integer $N$ there exists a unique \textbf{legal} decomposition $\sum_{i=1}^{m} {a_i G_{m+1-i}}$ with $a_1>0$, the other $a_i \ge 0$, and one of the following two conditions, which define a legal decomposition, holds.
\begin{itemize}
\item We have $m < L$ and $a_i=c_i$ for $1\le i\le m$.
\item There exists an $s\in\{1,\dots, L\}$ such that \be\label{eq:legalcondition2} a_1\ = \ c_1,\ a_2\ = \ c_2,\ \dots,\ a_{s-1}\ = \ c_{s-1}\ {\rm{and}}\ a_s<c_s, \ee $a_{s+1}, \dots, a_{s+\ell} =  0$ for some $\ell \ge 0$, and $\{b_i\}_{i=1}^{m-s-\ell}$ (with $b_i = a_{s+\ell+i}$) is either legal or empty.
\end{itemize} There exist constants $C_{{\rm Lek}} > 0$ and $d$ such that as $n \to\infty$ the average number of summands in a generalized Zeckendorf decomposition of integers in $[G_n, G_{n+1})$ is $C_{{\rm Lek}}n + d + o(1)$.
\end{thm}

The above theorem can be generalized. The decompositions above involve only non-negative summands, and if a $c_i \ge 2$ we may have multiple copies of a summand in a decomposition. We may ask what happens if we allow negative summands. For example, for the Fibonacci numbers each potential summand has a coefficient of 0, 1 or -1. Alpert \cite{Al} proved a unique decomposition exists again, with the non-adjacency condition becoming the gap between opposite signed summands must be at least 3, and between same signed summands must be at least 4. These are called far-difference representations, and can be generalized to other signed sequences (see \cite{DDKMV}). See also \cite{DDKMMV} for other generalizations of the notion of a legal decomposition.

After determining the mean number of summands in our decompositions, the next question is the variance or, more generally, the distribution of the fluctuations about the mean. Note the number of summands includes multiplicities; thus if $m = 1 G_{701} + 24 G_{601} + 2013 G_2$ is a legal decomposition, then there are $1+24+2013 = 2038$ summands. Using techniques from ergodic theory and number theory  the fluctuations about the mean were shown to converge to a Gaussian (see \cite{DG,FGNPT,GTNP,LT,Ste1}). Using a more combinatorial perspective, Kolo$\breve{{\rm g}}$lu, Kopp, Miller and Wang \cite{KKMW, MW1, MW2} reproved these results for the positive linear recurrences studied above, and their proof applies to the far-difference representations as well (see \cite{DDKMV}). In the case of the Fibonacci sequence, their method reproved Zeckendorf's theorem by partitioning the integers in $[F_n, F_{n+1})$ by the number of summands in their decomposition, obtaining a closed form expression for this using the cookie or stars and bars problems (the percentage of integers in $[F_n, F_{n+1})$ with exactly $k+1$ summands is $\ncr{n-k-1}{k}/F_{n-1}$, which by Stirling's formula converges to the Gaussian). Explicitly, we have

\begin{thm}[Gaussian Behavior of Summands in Generalized Zeckendorf Decompositions]\label{thm:gaussianzeck} Let $\{G_n\}$ be a positive linear recurrence as in Theorem \ref{thm:legaldecomp}. For each $m \in [G_n, G_{n+1})$ let $k(m)$ be the number of summands in $m$'s generalized Zeckendorf decomposition. As $n \to \infty$ the distribution of the $k(m)$'s for $m \in [G_m, G_{n+1})$ converges to a Gaussian with explicitly computable mean of order $C_{{\rm Lek}}n$ (for some $C_{{\rm Lek}}>0$) and variance of order $n$.
\end{thm}

\begin{rek} When we say the number of summands converges to a Gaussian this means that as $n\to\infty$ the fraction of $m \in [G_n, G_{n+1})$ such that the number of summands in $m$'s Zeckendorf decomposition is in $[\mu_n - a\sigma_n, \mu_n + b\sigma_n]$ converges to $\frac1{\sqrt{2\pi}} \int_a^b e^{-t^2/2}dt$, where $\mu_n$ is the mean number of summands for $m \in [G_n, G_{n+1})$ and $\sigma_n^2$ is the variance. \end{rek}

\subsection{Notation}

We now turn to the main object of study of this paper, the distribution of gaps between summands in generalized Zeckendorf decompositions. Though the actual combinatorial approach used in \cite{MW1,MW2} is not directly applicable here, the idea of partitioning based on a desired property is, and leads to very tractable expressions for the desired quantities.

Before stating our results we first set some notation and recall a needed result. Let $\{G_n\}$ be a positive linear recurrence, so every positive integer has a unique legal decomposition whose summands are elements of this sequence. We consider $m \in [G_n, G_{n+1})$. From the definition of $\{G_n\}$ we see that $G_n$ must be a summand in the decomposition of $m$ (if not, the largest possible combination would be too small to be in $[G_n, G_{n+1})$), though if the coefficient $c_1$ in the defining recurrence of $G_n$ is greater than 1 then it is possible to have multiple copies of $G_n$ in $m$'s decomposition. We can therefore write $m$ as \be\label{eq:decompminGn} m \ = \  \sum_{j=1}^{k(m)} G_{r_j} \ \ \ (r_{k(m)} \ = \ n). \ee Returning to our previous example of $m = 1 G_{701} + 24 G_{601} + 2013 G_2$, we find 2035 gaps of length 0 (2012 coming from $2013G_2$ and 23 from $24G_{601}$), one gap of length 599 (coming from $G_{601}$ and $G_{2}$), and one gap of length 100 (from $G_{701}$ and $G_{601}$). By Theorem \ref{thm:gaussianzeck} the $k(m)$'s converge to being normally distributed with mean of order $n$ and standard deviation of order $\sqrt{n}$; in particular, most $k(m)$'s are close, on an absolute scale, to the mean.

\begin{itemize}

\item \emph{Spacing gap measure:} We define the spacing gap measure of an $m \in [G_n, G_{n+1})$ with $k(m)$ summands by  \bea \nu_{m;n}(x)  \ := \  \frac1{k(m)-1} \sum_{j=2}^{k(m)} \delta\left(x - (r_j - r_{j-1})\right). \eea Note we are not including the gap to the first summand, as this is not a gap \emph{between} summands; as the typical $k(m)$ is growing the contribution of one extra gap is negligible in the limit, and it is technically cleaner.

\item \emph{Average spacing gap measure:} If $k(m)$ is the number of summands in $m$'s generalized Zeckendorf decomposition, then it has $k(m)-1$ gaps. Thus the total number of gaps for all $m \in [G_n, G_{n+1})$ is \be N_{{\rm gaps}}(n) \ := \ \sum_{m=G_n}^{G_{n+1}-1} \left(k(m) - 1\right), \ee and by the Generalized Lekkerkerker Theorem we have \bea N_{{\rm gaps}}(n) & \ = \ & \left(C_{{\rm Lek}} n + d + o(1) - 1\right)  \cdot \left(G_{n+1} - G_n\right)\nonumber\\ & \ = \ & C_{{\rm Lek}} n \left(G_{n+1} - G_n\right)  + O\left(G_{n+1} - G_n\right). \eea We define the average spacing gap measure for all $m \in [G_n, G_{n+1})$ by \bea \nu_n(x) & \  := \ & \frac1{N_{{\rm gaps}}(n)} \sum_{m=G_n}^{G_{n+1}-1} \sum_{j=2}^{k(m)} \delta\left(x - (r_j - r_{j-1})\right)\nonumber\\ & \ = \ & \frac1{N_{{\rm gaps}}(n)} \sum_{m=G_n}^{G_{n+1}-1} \left(k(m)-1\right) \nu_{m;n}(x). \eea Equivalently, if we let $P_n(k)$ denote the probability of getting a gap of length $k$ among all gaps from the decompositions of all $m \in [G_n, G_{n+1})$, then \be \nu_n(x) \ = \ \sum_{k=0}^{n-1} P_n(k) \delta(x - k). \ee

\item \emph{Limiting average spacing gap measure, limiting gap probabilities:} If the limits exist, we let \be \nu(x) \ = \ \lim_{n\to\infty} \nu_n(x), \ \ \ \ P(k) \ = \ \lim_{n \to \infty} P_n(k).\ee One of our main results is to prove these limits do exist, and determine them. While there has been some previous work on the average gap measures, the limiting behavior of individual gap measures has not been studied before; we do so below, and prove that they almost surely converge to the average measure.

\item \emph{Longest gap:} Given a decomposition $m = G_{r_1} + G_{r_2} + \cdots + G_{r_{k(m)}}$ for $m \in [G_n, G_{n+1})$, the longest gap, denoted $L_n(m)$, is the maximum difference between adjacent indices in $m$'s decomposition. Thus $L_n(m) := \max_{2 \le j \le k(m)} \left|r_j - r_{j-1}\right|$.

\end{itemize}

We need one last item before we can state our first results. Recall Binet's formula gives a closed form expression for the $n$\textsuperscript{th} Fibonacci number, specifically it equals \be \frac{1+\sqrt{5}}{2\sqrt{5}} \left(\frac{1+\sqrt{5}}{2}\right)^n - \frac{1-\sqrt{5}}{2\sqrt{5}} \left(\frac{1-\sqrt{5}}{2}\right)^n\ee (note the above expression is a little different than the standard realization of Binet's formula; this is due to the fact that our Fibonacci sequence has the indices of all terms shifted by 1). Here $(1 \pm \sqrt{5})/2$ are the two roots to the associated characteristic polynomial of the Fibonacci recurrence; as the first root is larger than 1 in absolute value and the second is less than 1 in absolute value, for large $n$ the $n$\textsuperscript{th} Fibonacci number is approximately the first summand. The following lemma is standard. It essentially follows immediately from the Perron-Frobenius Theorem for irreducible matrices and some additional algebra (though it can be proved directly, which is done in Appendix A of \cite{BBGILMT}).

\begin{lem}[Generalized Binet's Formula]\label{lem:genbinetf} Consider the positive linear recurrence \be G_{n+1} \ = \ c_1 G_n + c_2 G_{n-1} + \cdots + c_{L} G_{n+1-L}  \ee with the $c_i$'s non-negative integers and $c_1, c_{L} > 0$. Let $\lambda_1, \dots, \lambda_L$ be the roots of the characteristic polynomial  \be f(x) \ := \ x^{L} - \left(c_1 x^{L-1} + c_2 x^{L-2} + \cdots + c_{L-1} x + c_L\right) \ = \ 0, \ee ordered so that $\left|\lambda_1\right| \ge \left|\lambda_2\right| \ge \cdots \ge \left|\lambda_L\right|$. Then $\lambda_1 > \left|\lambda_2\right| \ge \cdots \ge \left|\lambda_L\right|$, $\lambda_1 > 1$ is the unique positive root, and there exist constants such that \be G_n \ = \ a_1 \lambda_1^n + O\left(n^{L-2} \lambda_2^n\right). \ee More precisely, if $\lambda_1, \omega_2, \dots, \omega_r$ denote the distinct roots of the characteristic polynomial with multiplicities 1, $m_2, \dots, m_r$, then there are constants $a_1 > 0, a_{i,j}$ such that \be G_n \ = \ a_1 \lambda_1^n + \sum_{i=2}^r \sum_{j=1}^{m_i} a_{i,j} n^{j-1} \omega_i^n. \ee \end{lem}

\subsection{Results: Gaps in the Bulk}\ \\

We can now state our results for gaps in the bulk.

\begin{thm}[Average Gap Measure in the Bulk]\label{thm:avegapsbulk} Let $\{G_n\}$ be a positive linear recurrence of length $L$ as in Theorem \ref{thm:legaldecomp}, with the additional constraint that each $c_i \ge 1$. Let $\lambda_1 > 1$ denote the largest root (in absolute value) of the characteristic polynomial of the $G_n$'s, and let $a_1$ be the leading coefficient in the Generalized Binet expansion (thus $G_n = a_1 \lambda_1^n + o(\lambda_1^n)$). Let $P_n(k)$ be the probability of having a gap of length $k$ among the decompositions of $m \in [G_n, G_{n+1})$, and let $P(k) = \lim_{n\to\infty} P_n(k)$. Then
\be \threecase{P(k) \ = \ }{1 - (\frac{a_1}{C_{{\rm Lek}}})(2\gl_1^{-1} + a_1^{-1} - 3)}{if $k = 0$}{\gl_1^{-1}(\frac{1}{C_{{\rm Lek}}})(\gl_1(1-2a_1) + a_1)}{if $k=1$}{(\gl_1-1)^2 \left(\frac{a_1}{C_{{\rm Lek}}}\right) \gl_1^{-k}}{if $k \ge 2$.}
\ee In particular, the probability of having a gap of length $k \ge 2$ decays geometrically, with decay constant the largest root of the characteristic polynomial.
\end{thm}

We included the condition $c_i \ge 1$ above to simplify the algebra. An analogue of the above theorem holds for general positive linear recurrences, but the counting becomes more involved and it is not as easy to extract nice closed form expressions. For such recurrences it is clear that there is geometric decay for gaps larger than the recurrence length $L$, but the behavior for $k < L$ depends greatly on which $c_i$'s vanish.

We isolate some important examples.

\begin{cor} The following hold.

\begin{itemize}

\item For base $B$ decompositions, $P(0) =(B-1)(B-2)/B^2$, and for $k \geq 1$, $P(k) = c_B B^{-k}$, with $c_B = (B-1)(3B-2)/B^2$.

\item For Zeckendorf decompositions, $P(k) = 1/\varphi^{k}$ for $k \ge 2$, with $\varphi = \frac{1+ \sqrt{5}}{2}$ the golden mean.
\end{itemize}
\end{cor}

The proof of Theorem \ref{thm:avegapsbulk} falls from a careful counting of the number of times each gap length occurs. For $k \ge 0$ let \be X_{i,i+k}(n)\ =\  \#\{m \in [G_n, G_{n+1}):\ \text{$G_i$, $G_{i+k}$ in $m$'s decomposition, but not $G_{i+q}$ for $0 < q < k$}\}. \ee Note we can deduce the $k=0$ behavior if we know the answer for each $k \ge 1$. Then \begin{equation}
P(k)\ =\ \lim_{n \to \infty} \frac{ \sum_{i=1}^{n-k} X_{i, i+k}(n)}{N_{{\rm gaps}} (n)}.
\end{equation}
The denominator is well-understood by Lekkerkerker's theorem; the proof of Theorem \ref{thm:avegapsbulk} follows from a good analysis of $X_{i,i+k}(n)$, which can be deduced from the uniqueness arguments in the generalized Zeckendorf decompositions. We give the proof in \S\ref{sec:gapsinthebulkI}.

Theorem \ref{thm:avegapsbulk} describes the limiting behavior of the \emph{average} of the individual gap measures $\nu_{m;n}(x)$. As $n\to\infty$, for almost all $m \in [G_n, G_{n+1})$ the \emph{individual} measures $\nu_{m;n}(x)$ are close to the average gap measure.

\begin{thm}[Individual Gap Measure in the Bulk]\label{thm:indgapmeasurebulk} Let $\{G_n\}$ be a positive linear recurrence as in Theorem \ref{thm:legaldecomp}, with the additional assumption that each $c_i \ge 1$. As $n\to\infty$, the individual gap measures $\nu_{m;n}(x)$ converge almost surely in distribution\footnote{A sequence of random variables $R_1, R_2, \dots$ with corresponding cumulative distribution functions $F_1, F_2, \dots$ \emph{converges in distribution} to a random variable $R$ with cumulative distribution $F$ if $\lim_{n\to\infty} F_n(r) = F(r)$ for each $r$ where $F$ is continuous.} to the limiting gap measure from Theorem \ref{thm:avegapsbulk}.\end{thm}

We quickly sketch the main ideas of the proof. Let $\widehat{\nu_{m;n}}(t)$ denote the characteristic function\footnote{The characteristic function of a random variable $X$ is $\E[e^{itX}]$, with a similar definition for a measure. We denote the characteristic function of a measure $\mu$ by $\widehat{\mu}$, as it is the Fourier transform of the measure (up to a normalization constant).} of $\nu_{m;n}(x)$, and $\widehat{\nu}(t)$ the characteristic function of the average gap distribution from Theorem \ref{thm:avegapsbulk}. Let $\E_m[\cdots]$ denote the expectation over all $m \in [G_n, G_{n+1})$, with all $m$ equally likely to be chosen. We first show that $\lim_{n\to\infty}\E_m[\hnumnt]$ equals $\hnut$, and then we show $\lim_{n\to\infty}\left[\left(\hnumnt - \hnut\right)^2\right] = 0$. This allows us to invoke L\'evy's continuity theorem to obtain convergence in distribution for almost all $m \in [G_n, G_{n+1})$ as $n\to\infty$. The key steps in the proof are to replace $k(m)$ with its average (and use the Gaussianity results to control the error), and introduce more general indicator functions such as $X_{i,i+g_1,j,j+g_2}(n)$, reducing the proof to a counting problem.


\subsection{Results: Longest Gap}\ \\

Our first two results were for gaps in the bulk. Given each $m \in [G_n, G_{n+1})$ we associated a sequence of gaps, which we either analyzed individually for each $m$ or amalgamated and did all $m$ simultaneously. Another natural problem to investigate is the distribution of the largest gap between summands for each such $m$. Specifically, let \be L_n(m) \ := \ \max_{2 \le j \le k(m)} (r_j - r_{j-1}), \ee where as always $k(m)$ is the number of summands in the decomposition of $m$, and the $G_{i_\ell}$'s are the summands (see \eqref{eq:decompminGn}).

If $G_{n+1} = 2 G_n$, then the distribution of $L_n(m)$ for $m \in [G_n, G_{n+1})$ is essentially that of the longest run of consecutive tails in $n$ tosses of a fair coin whose final toss is a head. The answer for coins is well-known, both for fair and biased coins (see for example \cite{Sch}). What is particularly remarkable about the coin toss problem is how tightly concentrated the answer is about the mean. For a coin with probability $p$ of heads and $q=1-p$ of tails, the expected longest run of heads is \be\label{eq:schillingformula} \log_{1/p}(nq) - \frac{\gamma}{\log p} - \frac12 + r_1(n) + \epsilon_1(n) \ = \ \frac{\log(nq)}{\log(1/p)} + \frac{\gamma}{\log(1/p)} - \frac12 + r_1(n) + \epsilon_1(n)  \ee while the variance is \be \frac{\pi^2}{6 \log^2 p} + \frac1{12} + r_2(n) + \epsilon_2(n), \ee where $\gamma$ is Euler's constant, the $r_i(n)$ are at most .000016, and the $\epsilon_i(n)$ tend to zero as $n\to\infty$. Very importantly, the variance is bounded independent of $n$ (by essentially 3.5). This implies that there is essentially no fluctuation of the observed longest string of heads. We find similar behavior, both in terms of the logarithmic size of the longest term in our sequence as well as the strong concentration about the average.

Before we can state our results, however, we need to introduce some notation. It is technically more convenient to rewrite the recurrence relation where we only record the \textbf{non-zero} coefficients. \emph{Thus, in the sections on longest gaps, we write our positive linear recurrence as}
\beq\label{eq:recurrenceforlongestgaps} G_{n+1}\ =\ c_{j_1 + 1} G_{n- j_1 } + c_{j_2+1} G_{n- j_2} + \cdots + c_{j_L+1} G_{n-j_{L}},\eeq
where $j_1 = 0$, $j_1 < j_2 < \cdots < j_L$, and all recurrence coefficients not shown are zero.

\begin{defn} We use the following notation below.

\begin{itemize}
\item Gaps in the recurrence: Set $g_{i-1} = j_{i}  - j_{i-1}$, with the convention that $g_0 = 1$.

\item Associated polynomials: The following polynomials, arising from the recurrence relation for the $G_i$'s, are useful in computing the generating function for the longest gap: \bea \calM(s) &\ =\ & 1 - c_{1} s - c_{j_2 + 1} s^{j_2 + 1} - \cdots - c_{j_L + 1} s^{j_L +1} \nonumber\\ \calR(s) & = & c_{1} + c_{j_2+1} s^{j_2} + \cdots + (c_{j_L+1}-1) s^{j_L} \nonumber\\ \calG(s) & =& -\calM(s) / (s-1/\lambda_1). \eea To simplify the analysis, we always assume the polynomials $\calM(s)$ and $\calR(s)$ have no multiple roots, and no roots of absolute value 1; these assumptions hold in many cases of interest (such as the Fibonacci numbers). Note that the roots of $\calM(s)$ are the reciprocals of the roots of the polynomial associated to the original recurrence relation, and thus our assumption implies that polynomial also does not have multiple roots or roots of absolute value 1.  Since $1/\lambda_1$ is a root of $\calM(s)$, we have that $\calG(s)$ is a polynomial.
\end{itemize}
\end{defn}

By Zeckendorf's theorem for $PLRS$ sequences (Theorem \ref{thm:legaldecomp}), any $m \in [G_n, G_{n+1})$ has a unique legal decomposition $m = a_{i_1 }G_{i_1} + \cdots+ a_{i_r} G_{i_r}$, with $i_1 = n$ and $i_1 > i_2 > \cdots > i_r$. Below we determine the asymptotic behavior of the longest gaps.

\begin{thm}[Longest gap in generalized Zeckendorf expansions]\label{thm:longestgap}  Let $\{G_n\}$ be a positive linear recurrence as in Theorem \ref{thm:legaldecomp}, and assume the associated polynomials $\calM(s)$ and $\calR(s)$ do not have multiple roots or roots of absolute value 1. Let \be P(n,f)\ :=\ \frac{\#\{m \in [G_n, G_{n+1}): L_n(m) < f\}}{G_{n+1}-G_n}\ee be the cumulative distribution of the longest gap in the Zeckendorf decompositions of $m \in [G_n, G_{n+1})$; note we are computing gaps \texttt{strictly less than} $f$, and we do not include the gap in the beginning.


\begin{enumerate}

\item Asymptotically we have
\bea
P(n,f) \ =\     \exp\left( - n \lambda_1^{- f} \frac{ \lambda_1 \calR(\frac{1}{\lambda_1})}{ \calG(\frac{1}{\lambda_1})} \right)
+ O\left(n f \left(\frac{ R_{\min}}{\lambda_1} \right) ^f  +  n  \left( \frac{1}{\lambda_1} \right)^{2f}
+ f  \left( \frac{1}   {\lambda_1 R_{\min }} \right)^n  \right), \eea where  $\lambda_1$ is the greatest eigenvalue of the recurrence relation for $G_n$, and $R_{\min} \in \R$ is any constant with $\lambda_1 < R_{\min}  < 1$.

\item Let $K = \lambda_1 \calR(1/\lambda_1) / \calG(1/\lambda_1)$ and $\gamma$ be Euler's constant. The mean of the longest gap, $\mu_{n}$, and the variance of the longest gap, $\sigma_n^2$, are given by \bea \mu_{n} & \  = \ &   \frac{ \log \left(nK\right) }{ \log
\lambda_1}  + \frac{   \gamma}{ \log \lambda_1}  - \frac{1}{2} + o(1) \nonumber\\   \sigma_n^2 &=& \frac{\pi^2} { 6 (\log \lambda_1)^2} + o(1).
\eea

\end{enumerate}

\end{thm}

The proof proceeds by introducing a generating function for the longest gap distribution, where we obtain the probabilities by analyzing the cumulative distribution function. We use a partial fraction decomposition to extract information from the generating function, and use Rouche's theorem (among others) to deal with the technicalities that arise.

The fit between numerics and theory is excellent (and in fact these experiments were crucial in helping us confirm our analysis). For example, consider the Fibonacci numbers. We chose 100 numbers randomly from $[F_n, F_{n+1})$ with $n = 1,000,000$. We observed a mean of 28.51 and a standard deviation of 2.64, which compares very well with the predictions of 28.73 and 2.67. Increasing $n$ to 10,000,000 and looking at 20 randomly chosen numbers yielded a mean of 33.6 and a standard deviation of 2.33, again close to the predictions of 33.52 and 2.665.

\begin{rek} The Fibonacci case is the easiest to analyze, but it took a few approaches to determine the most efficient way to compute these quantities. Due to costs to store and recall objects from memory, as well as costs to use the Binet formula, we found it was best to just use Binet's formula to find $F_n$ and $F_{n+1}$, and then use the recurrence relation to compute backwards. We then tested each number, as it was computed, to see if it was in the Zeckendorf decomposition of our randomly chosen interval $[F_n, F_{n+1})$. This was significantly faster than using Binet's formula to find the largest Fibonacci number less than our number, even though occasionally we computed numbers we did not need.

We saw similar behavior in other recurrences, though their notion of legal decompositions lead to slightly more complicated programs. For example, we studied $a_{n+1} = 2a_n + 4 a_{n-1}$. When $n=51,200$ (respectively $102,400$) the predicted mean was 9.95 (resp. 10.54) and the standard deviation was 1.09; choosing 100 points randomly in the interval yielded a mean of 9.91 (resp. 10.45) and a standard deviation of 1.22 (resp. 1.10), very much in line with the predictions.
\end{rek}

\begin{rek} In our investigations of the longest gap, it was slightly more convenient to first investigate quantities associated to the longest gap being \texttt{less than} $f$, and then trivially pass to being \texttt{at most} $f$. \end{rek}

\begin{rek}\label{rek:schilling} We can compare our predicted formula for $\mu_n$ from Theorem \ref{thm:longestgap} with previous work on the length of gaps between heads in tossing a fair coin. Taking $p=1/2$ in \eqref{eq:schillingformula}, we are studying the recurrence $G_{n+1} = 2 G_n$. We find $\lambda_1 = 2$, and after some algebra we get a main term of approximately $\frac{\log n}{\log \lambda_1} + \frac{\gamma}{\log \lambda_1} - \frac12$. The only difference is that we have $\log(n)$ instead of $\log(n/2)$; however the effect of the additional factor of 2 is to shift the mean down by 1. The reason our answer does not precisely match this case is that we are studying a slightly different quantity, as we are not considering the length of the initial segment.
\end{rek}


\subsection{Structure of the paper.} The paper is organized as follows. In \S\ref{sec:gapsinthebulkI} we prove Theorem \ref{thm:avegapsbulk} for the average gap measure in the bulk, and then prove the almost sure convergence for the individual measures in \S\ref{sec:gapsinthebulkII}. We then prove Theorem \ref{thm:longestgap} in \S\ref{sec:longestgap}, and conclude with some final remarks.

\section{Gaps in the Bulk I: Average Behavior}\label{sec:gapsinthebulkI}

In this section we prove Theorem \ref{thm:avegapsbulk}. Our combinatorial approach begins by computing $X_{i,i+k}(n)$, which allows us to find $P_n(k)$. We can determine $\Xn$ by counting the number of choices of the summands $\{G_1, G_2, \dots, G_n\}$ such that $G_i, G_{i+k}$ and $G_n$ are chosen, no summand whose index is between $i$ and $i + k$ is chosen, and all other indices are free to be chosen subject to the requirement that we have a legal decomposition. Let $\Ln$ and $\Rn$ be the number of ways to choose a valid subset of summands from those before the gap of length $k$ starting at $G_i$ and after the gap (respectively). Since \be G_{j+1} \ = \ c_1 G_{j} + \cdots + c_L G_{j+1-L} \ee where $c_i \geq 1$, any time we have a gap of length $k > 1,$ the recurrence `resets' itself. We see that $\Ln$ and $\Rn$ are independent of each other when $k \ge 2$; thus for $k \ge 2$ we have \be \Xn\ = \ \Ln \cdot \Rn. \ee The behavior for $k \le 1$ is more delicate due to the dependencies, but follows from a careful counting.

We have the following counting lemma.

\begin{lem}\label{lem:countlem1} Let $\{G_n\}$ be a positive linear recurrence as in Theorem \ref{thm:legaldecomp} with each $c_i \ge 1$. Consider all $m \in [G_n, G_{n+1})$ with a gap of length $k \geq 2$ starting at $G_i$ for $1\leq i \leq n-k$. The number of valid choices for subsets of summands before the gap, $\Ln$, is
\be
\Ln \ = \
G_{i+1} - G_{i},
\ee
while the number of valid choices for subsets of summands after the gap, $\Rn$, is
\be
\Rn \ = \
G_{n-i-k+2} - 2G_{n - i - k + 1} +G_{n-i-k}.
\ee
\end{lem}

\begin{proof}
To count $\Ln$, we count the number of ways to have a legal decomposition that must have the summand $G_i$ such that all other summands which are less than $G_i$ are free to be chosen or not. It is very important that $k \geq 2$, as this means the summand at $G_{i+k}$ does not interact with the summands earlier than $G_i$ through the recurrence relation. Thus $\Ln$ is the same as the number of legal choices of summands from $\{G_1, G_2, \dots, G_i\}$ with $G_i$ chosen. As each integer in $[G_i, G_{i+1})$ has a unique legal decomposition with $G_i$ chosen, we see $\Ln$ equals the number of elements in this interval, which is just $G_{i+1} - G_i$.

To compute $\Rn$, we need to consider how many ways we can choose summands from $\{G_{i+k}$, $G_{i+k+1}$, $\dots$, $G_n\}$ such that $G_{i+k}$ and $G_n$ are chosen and the resulting decomposition is legal; since $k \geq 2$ the summands from $G_i$ and earlier cannot affect our choices here. Thus our problem is equivalent to asking how many legal ways there are to choose summands from $\{G_1, G_2, \dots, G_{n-i-k+1}\}$ with $G_1, G_{n-i-k+1}$ both chosen and the rest free. There are many ways to compute this; the simplest is to note that this equals the number of legal choices with $G_{n-i-k+1}$ chosen and where we \emph{may or may not choose $G_1$}, minus the number of legal choices with $G_{n-i-k+1}$ chosen where we \emph{do not choose $G_1$}. By a similar argument as above, the first count is $G_{n-i-k+2}-G_{n-i-k+1}$ (as it is the number of legal decompositions of a number in $[G_{n-i-k+1}, G_{n-i-k+2})$), while the second is $G_{n-i-k+1} - G_{n-i-k}$. The proof is completed by subtracting. \end{proof}

We also need a way for counting how many legal decompositions have a gap of length one, which is given by the following lemma. The main idea of the proof is to remove the dependencies by breaking into cases and then arguing as above.

\begin{lem}\label{lem:countlem2} Let $\{G_n\}$ be a positive linear recurrence as in Theorem \ref{thm:legaldecomp} such that $c_i \ge 1$. Consider all $m \in [G_n, G_{n+1})$ with a gap of length 1 starting at $G_i$ for $1\leq i \leq n-1$. Then \be \Xnn \ = \ (G_{n+1}-G_n) -G_{i+1}(G_{n-i}-G_{n-i-1}) - G_{i}(G_{n-i+1}-2G_{n-i}+G_{n-i-1}). \ee
\end{lem}

\begin{proof} We cannot count as in Lemma \ref{lem:countlem1}, since $L_{i,i+1}(n)$ and $R_{i,i+1}$ are no longer independent. Instead, we consider the total number of decompositions in $[G_n, G_{n+1})$ (which is $G_{n+1} - G_n$) and subtract off the three different ways to \emph{not} have a gap of length one starting at $G_i$ for a decomposition: (1) not including $G_i$ and not including $G_{i+1}$, (2) including $G_i$ but not including $G_{i+1}$, and finally (3) not including $G_i$ but including $G_{i+1}$. In each of these three cases, we can use the methods of Lemma \ref{lem:countlem1} since there are no dependency issues. Note the last two cases are very similar.

We do the first case in detail; the other two cases follow similarly. If we have neither $G_i$ nor $G_{i+1}$ then this is the same as counting how many ways there are to choose legal combinations from $\{G_1, \dots, G_{i-1}\}$ where all the $G_{j}$'s are free, times the number of ways to choose legal combinations from $\{G_{i+2}, \dots, G_n\}$ with $G_n$ taken and all others free; we multiply the two answers as we have independence due to the fact that the gap is at least 2. The latter is easy, as it is the same as choosing legal combinations from $\{G_1, \dots, G_{n-i-1}\}$ with $G_{n-i-1}$ chosen, which is just $G_{n-i}-G_{n-i-1}$ (as this is equivalent to the number of integers in $[G_{n-i-1}, G_{n-i})$). For the former, if $\ell$ is the largest index chosen then there are $G_{\ell+1} - G_\ell$ choices. We sum from $\ell = 1$ to $i-1$ and get a telescoping sum which equals $G_i - 1$. We then add 1 to count the case where no index is chosen, giving $G_i$. Thus the number of integers in this case is the product $G_i (G_{n-i}-G_{n-i-1})$.

If we have $G_i$ but not $G_{i+1}$, then a similar argument gives $G_{i+1}-G_i$ choices for the left part and $G_{n-i}-G_{n-i-1}$ for the right, and thus the total number of choices is $(G_{i+1}-G_i) (G_{n-i}-G_{n-i-1})$.

Finally, if we have $G_{i+1}$ but not $G_i$ then the number of combinations for the left is again $G_i$. For the right, we look at the number of ways to choose from $\{G_{i+1}, \dots, G_n\}$ with the first and last chosen; this is equivalent to choosing from $\{G_1, \dots, G_{n-i}\}$ with the first and last chosen. If the first were free there would be $G_{n-i+1}-G_{n-i}$, while if the first were not chosen it would be as if we shifted all indices down by one, giving $G_{n-i}-G_{n-i-1}$. Thus if we subtract the second from the first we get our answer of $G_{n-i+1}-2G_{n-i}+G_{n-i-1}$, which we then multiply by $G_i$ to get $G_i(G_{n-i+1}-2G_{n-i}+G_{n-i-1})$.

The proof is completed by adding the three cases and subtracting this from $G_{n+1}-G_n$. \end{proof}

We now prove Theorem \ref{thm:avegapsbulk}. We use little-oh and big-Oh notation for the lower order terms, which do not matter in the limit. If \be \lim_{x \to \infty} \frac{F(x)}{G(x)} \ = \ 0, \ee we write $F(x) = o(G(x))$ and say $F$ is little-oh of $G$, while if there exist $M,x_0 > 0$ such that
$\left|F(x)\right| \leq M G(x)$ for all $x > x_0$ we write $F(x) = O(G(x))$ and say $F$ is big-oh of $G$. In particular, $o(1)$ represents a term that decays to zero as $n\to\infty$, while $O(1)$ represents a term bounded by a constant.

\begin{proof}[Proof of Theorem \ref{thm:avegapsbulk}] There are three cases to consider: $k=0$, $k=1$ and $k\geq2$. When $k \ge 1$ we use the generalized Binet's formula and take limits. When $k=0$ it is harder to count gaps of length 0 since a decomposition could have multiple gaps of length 0 at $G_i$; fortunately we can deduce the number of these gaps by knowing the number of gaps with $k \ge 1$.

As our analysis of gaps of length $k$ had different answers for $k=1$ and $k\ge 2$, we first consider the case when $k \ge 2$. We need to compute \be
P(k)\ = \ \lim_{n \to \infty} \frac{ \sum_{i=1}^{n-k} X_{i, i+k}(n)}{N_{{\rm gaps}} (n)}.
\ee
By Lemma \ref{lem:countlem1},
\be \Xn \ = \  \Ln \cdot \Rn \ = \ \left(G_{i+1}-G_{i}\right) \cdot \left(G_{n-i-k+2} - 2G_{n - i - k + 1} +G_{n-i-k}\right), \ee
and by Lemma \ref{lem:genbinetf}, \be G_i \ = \ a_1 \lambda_1^i + O\left(i^{L-2} \lambda_2^i\right) \ = \ a_1 \lambda_1^i \left(1 + O\left(i^{L-2} (\lambda_2/\lambda_1)^i\right)\right). \ee We want to use little-oh notation for the error term above; unfortunately the error is not necessarily small if $i$ is close to $0$. The error is $o(1)$ if $i$ is at least $\log^2 n$ and is bounded for smaller $i$. Thus we introduce the notation $o_{i;n}(1)$ for an error that is $o(1)$ for $i \ge \log^2 n$ and bounded otherwise.

Thus \bea \Xn & \ = \ & a_1 \lambda_1^i (\lambda_1 - 1) \left(1 + o_{i;n}(1)\right) \cdot a_1 \lambda_1^{n-i-k} (\lambda_1^2 - 2\lambda_1 + 1) \left(1 + o_{n-i-k;n}(1)\right) \nonumber\\ &=& a_1^2\lambda_1^{n-k} (\lambda_1 - 1)^3  \left(1 + o_{i;n}(1) + o_{n-i-k;n}(1)\right). \eea

As \be N_{{\rm gaps}} (n)\ =\ C_{{\rm Lek}} n \left(G_{n+1} - G_n\right)  + O\left(G_{n+1} - G_n\right) \ = \  C_{{\rm Lek}} \cdot n\cdot a_1\cdot \lambda_1^n (\lambda_1 - 1) + O\left(\lambda_1^{n}\right), \ee we find \bea P_n(k) & \ = \ & \frac{\sum_{i=1}^{n-k} X_{i,i+k}(n)}{N_{{\rm gaps}} (n)} \nonumber\\ &=& \frac{\sum_{i=1}^{n-k} a_1^2 \lambda_1^{n-k}(\lambda_1 - 1)^3 \left(1 + o_{i;n}(1) + o_{n-i-k;n}(1)\right)}{C_{{\rm Lek}} \cdot n\cdot a_1\cdot \lambda_1^n (\lambda_1 - 1) + O\left(\lambda_1^{n+1}\right)}\nonumber\\ & \ = \ & (\lambda_1 - 1)^2 \left(\frac{a_1 }{C_{{\rm Lek}}}\right) \lambda_1^{-k} \left(1 + o(1)\right), \eea as the sum over $i \le \log^2 n$ and $i \ge n-k-\log^2 n$ is negligible. By taking the limit, which clearly exists for each $n$ and each $k \geq 2$, we obtain the claimed expression for $P(k)$ for $k\geq 2$.\\

If $k=1$ we can use Lemma \ref{lem:countlem2} to evaluate $\Xnn$ and use a similar argument as in the $k\geq2$ case, which gives $P(1)$.\\

Finally when $k=0$, since probability distributions must sum to one, after some algebra we find
\bea
P(0)  &\ =\ & 1- \left(P(1)+\sum_{k=2}^{\infty} P(k) \right) \nonumber\\
&=&  1- \left(\frac{a_1}{C_{{\rm Lek}}}\right) \left(2\gl_1^{-1}+a_1^{-1}-3 \right),
\eea which completes the proof.
\end{proof}


\section{Gaps in the Bulk II: Individual Measures}\label{sec:gapsinthebulkII}

In this section we prove Theorem \ref{thm:indgapmeasurebulk}. Recall the spacing gap measure of $m \in [G_n, G_{n+1})$ with decomposition given in (\ref{eq:decompminGn}) with $k(m)$ summands is defined to be  \bea \nu_{m;n}(x)  \ = \  \frac1{k(m)-1} \sum_{j=2}^{k(m)} \delta\left(x - (r_j - r_{j-1})\right). \eea We first recall some notation.

\begin{itemize}

\item $\hnumnt$: The characteristic function of $\nu_{m;n}(x)$.

\item $\hnut$: The characteristic function of the average gap distribution $\nu(x)$ from Theorem \ref{thm:avegapsbulk}.

\item $\E_m[\cdots]$: The expected value over $m \in [G_n, G_{n+1})$ with the uniform measure; thus if $X: [G_n, G_{n+1})$ $\to$ $\R$ then \be \E_m[X] \ := \ \frac1{G_{n+1}-G_n} \sum_{m=G_n}^{G_{n+1}-1} X(m). \ee

\item $X_{j_1, j_1+g_1, j_2, j_2+g_2}$: The number of $m \in [G_n, G_{n+1})$ with a gap of length exactly $g_1$ starting at $j_1$ and a gap of length exactly $g_2$ starting at $j_2$ (we have suppressed the subscript $n$ as it is always understood from context). If $g_1$ or $g_2$ is zero then we count with multiplicity. For example, if $g_1 = 0$ and $g_2 = 3$ then an $m$ that has 5 summands at $G_{j_1}$ and has $G_{j_2}$ and $G_{j_2+3}$ (but no summands between these last two) is counted four times. We similarly count with multiplicity if we have $X_{j_1,j_1+g_1}$.

\end{itemize}

We sketch the proof. We use L\'evy's continuity theorem \cite{FG}, which says that if we have a sequence of random variables $\{R_r\}$ (which do not have to be defined on the same probability space) whose characteristic functions $\{\varphi_r\}$ converge pointwise to the characteristic function $\varphi$ of a random variable $R$, then the random variables $\{R_r\}$ converge in distribution to $R$ (i.e., the cumulative distribution functions of the $\{R_r\}$ converge to that of $\{R\}$ at all points of continuity).

Briefly, we show given any $\epsilon$ there is an $N_\epsilon$ such that for all $n \ge N_\epsilon$ the characteristic functions $\hnumnt$ are pointwise within $\epsilon$ for almost all $m \in [G_n, G_{n+1})$ (we can't have all the characteristic functions close, as some $m$ have very few gaps). Letting $\epsilon \to 0$ completes the proof that almost all individual measures converge pointwise.

Step 1 is to show that $\E_m[\hnumnt] = \hnut$. A key ingredient is to remove the individual normalizations of $\frac{1}{k(m)-1}$, where $k(m)$ is the number of summands in the generalized Zeckendorf decomposition of $m$; we can replace these with their average up to a negligible error term because of previous work on the Gaussian behavior of the number of summands. To complete the proof, we must show that most characteristic functions are concentrated near the mean. We do this in step 2 by showing $$\lim_{n\to\infty}\E_m\left[\left(\hnumnt - \hnut\right)^2\right]\ =\ 0,$$ which follows by reducing the problem to determining $X_{j_1,j_1+g_1,j_2,j_2+g_2}$.

\subsection{Expected Value of Individual Characteristic Functions}\label{sec:avemomentind}

The first step towards a proof of Theorem \ref{thm:indgapmeasurebulk} is to show that the expected value of the individual characteristic functions of the gap measures converge to the characteristic function of the average gap measure. Convergence in distribution follows from controlling the rate of convergence, which we handle in the next subsection.

\begin{pro}\label{prop:expectedvalue} Notation as above,  we have \be \lim_{n\to\infty} \E_m[\hnumnt]\ =\ \hnut. \ee
\end{pro}

We need some preliminary results before we can prove this proposition. Notice
\bea
\hnumnt & \ = \ & \int_{0}^{\infty} e^{ixt} \nu_{m;n}(x) dx \nonumber\\
& \ = \ & \km \sum_{j = 2}^{k(m)} e^{it(r_{j} - r_{j-1})},
\eea where \be m \ = \ G_{r_1} + G_{r_2} + \cdots + G_{r_{k(m)}}.\ee
Thus we have
\be
\label{expected}
\E_m[\hnumnt] \ =\  \frac{1}{G_{n+1}-G_n} \sum_{m=G_n}^{G_{n+1}-1} \km \sum_{j = 2}^{k(m)}  e^{it(r_{j} - r_{j-1})}.
\ee
The difficulty in evaluating $\E_m[\hnumnt]$ is that we must deal with the presence of the $k(m)-1$ factors. These vary with $m$, though weakly because of our Gaussianity result (Theorem \ref{thm:gaussianzeck}). As the mean is of order $n$ and the standard deviation is of order $\sqrt{n}$, the $k(m)$ are strongly concentrated about their mean. We first apply standard estimation arguments to show that we may safely replace $k(m)$ with its mean.

Notice that
\be\label{eq:whatkmisapprox}
\km \ = \ \frac{1}{\Cn} \ - \ \frac{(k(m)-1) \ - \ \left(\Cn\right)}{(k(m)-1) (\Cn)}.
\ee
We essentially replace $\km$ with $\frac{1}{\Cn}$ at a negligible cost, as the second factor above is extremely small most of the time and of moderate size almost never. We make this claim explicit in the next lemma.

\begin{lem}\label{lem:replacekm} Let $\{G_n\}$ be a positive linear recurrence as in Theorem \ref{thm:legaldecomp}, with each $c_i \ge 1$. Let $m\in [G_n, G_{n+1})$ have decomposition given by (\ref{eq:decompminGn}). Then for any fixed $t \ge 1$ we have \be \lim_{n \to \infty} \frac{1}{G_{n+1}-G_n} \sum_{m=G_n}^{G_{n+1}-1} \ \left(\frac{ (k(m)-1) \ - \ (\Cn)}{(k(m)-1) (\Cn) }\right) \sum_{j = 2}^{k(m)}e^{it(r_{j} - r_{j-1})} \ = \ 0, \ee  where $\Cn$ is the average number of summands needed in a decomposition for an integer in $[G_n, G_{n+1})$ in Theorem \ref{thm:legaldecomp}.
\end{lem}
\begin{proof}

The distribution of the number of summands in a decomposition for $m \in [G_n, G_{n+1})$ converges to a Gaussian by Theorem \ref{thm:gaussianzeck}. The average number of summands is $\Cn$ (with $C_{{\rm Lek}} > 0$) and the standard deviation is $b \sqrt{n} + o(\sqrt{n})$ for some $b > 0$. The proof is completed by partitioning based on the deviation of $k(m)$ from its expected value.

Fix a $\delta > 0$ and let \be I_n(\delta) \ := \  \left\{m \in [G_n, G_{n+1}): k(m) \in \left[C_{{\rm Lek}}n + d - b n^{1/2+\delta},\ \ C_{{\rm Lek}}n + d + b n^{1/2+\delta}\right]\right\}. \ee

\ \\

\textbf{Case 1: }
Let $m\in [G_n,G_{n+1}) \cap I_n(\delta)$; thus $k(m)$ is very close to $C_{{\rm Lek}}n \ + \ d +o(1)$. To simplify the expressions below remember we are writing $r_j$ and $r_{j-1}$ for the indices in the decomposition of $m$; while we should really write $r_j(m)$, as the meaning is clear we prefer this more compact notation. Therefore
\bea
& & \frac{1}{G_{n+1}-G_n} \sum_{m=G_n \atop m \in I_n(\delta)}^{G_{n+1}-1} \ \left(\frac{ (k(m)-1) \ - \ (\Cn)}{(k(m)-1)(\Cn)}\right) \sum_{j = 2}^{k(m)}e^{it(r_{j} - r_{j-1})} \nonumber\\
& \ll & \frac{1}{G_{n+1}-G_n} \sum_{m=G_n}^{G_{n+1}-1} \ \frac{n^{\frac{1}{2}+\delta}}{n^2} \ \sum_{j = 2}^{k(m)}e^{it(r_{j} - r_{j-1})} \nonumber\\
& \ll & \frac{1}{n^{\frac{3}{2} - \delta}} \frac{1}{G_{n+1}-G_n}\sum_{m=G_n}^{G_{n+1}-1} \sum_{j = 2}^{k(m)}e^{it(r_{j} - r_{j-1})} \nonumber \\
& \ll & \frac{1}{n^{\frac{3}{2} - \delta}} \frac{(G_{n+1}-G_n) n}{G_{n+1}-G_n} \ \ll \ n^{-1/2 + \delta}.
\eea

The passage from the third to the fourth line follows from $k(m) \ll n$; to see that, note there $n$ different indices, and each occurs at most $\max_{1 \le i \le L} c_i$ times (the $c_i$'s are the coefficients of the recurrence relation for the $G_n$'s).


\ \\

\textbf{Case 2:} Let $k(m) \notin I_n(\delta)$, which means $k(m)$ is not too close to $C_{{\rm Lek}}n + d$. Since the distribution of the number of summands needed for a decomposition converges to a Gaussian by Theorem \ref{thm:gaussianzeck}, for sufficiently large $n$ the probability of an $m \in [G_n,G_{n+1})$ such that $k(m)$ is more than $n^\delta$ standard deviations from the mean is essentially
\bea
2 \int_{b n^{\frac{1}{2}  + \delta}}^{\infty} \frac{1}{\sqrt{2\pi b^2n}}  \ e^{-t^2/2b^2n} dt \ \ll \  e^{-n^{2\delta}/2}.
\eea

Thus for sufficiently large $n$, the number of $m \in [G_n,G_{n+1})$ such that $k(m) \notin I_n(\delta)$ is essentially $(G_{n+1}-G_n) \cdot e^{-n^{2\delta}/2}$.

Therefore
\bea & &\frac{1}{G_{n+1}-G_n} \sum_{m=G_n,\atop m \notin I_n(\delta)}^{G_{n+1}-1} \ \left(\frac{ (k(m)-1) \ - \ (\Cn)}{(k(m)-1) (\Cn) }\right) \sum_{j = 2}^{k(m)}e^{it(r_{j} - r_{j-1})}  \nonumber\\
& \ll &  \frac{1}{G_{n+1}-G_n}\ e^{-n^{2\delta}/2}(G_{n+1}-G_n) \cdot \frac1{n} \cdot n \nonumber\\
&=& \ e^{-n^{2\delta}/2},\eea which tends rapidly to zero as $n\to\infty$. This completes the proof.
\end{proof}

\begin{rek}\label{rek:prodkmandapprox} In calculating the variance, we need to approximate $\left(\km \right)^2$. A similar argument shows that this can be replaced at a negligible cost with $\left(\frac{1}{\Cn}\right)^2$; the error in the resulting sums from these replacements is $o(1)$, and thus vanishes in the limit. \end{rek}

Proposition \ref{prop:expectedvalue} now follows immediately.

\begin{proof}[Proof of Proposition \ref{prop:expectedvalue}] We replace $\frac1{k(m)-1}$ with $\frac1{\Cn}$ in the argument below with negligible error by Lemma \ref{lem:replacekm}; this is desirable as we can now pull this factor outside of the $m$ summation. We have \bea \E_m[\hnumnt] & = & \frac{1}{G_{n+1}-G_n} \sum_{m=G_n}^{G_{n+1}-1} \km \sum_{j = 2}^{k(m)}e^{it(r_{j} - r_{j-1})} \nonumber \\
&=& \frac{1}{G_{n+1}-G_n}\ \frac{1}{\Cn} \sum_{m=G_n}^{G_{n+1}-1} \sum_{j = 2}^{k(m)}e^{it(r_{j} - r_{j-1})} + o(1) \nonumber \\
&=& \frac{1}{G_{n+1}-G_n}\ \frac{1}{\Cn}\sum_{g=0}^{n-1}  \sum_{j=1}^{n-g} X_{j,j+g}(n) e^{itg} \ + \ o(1) \nonumber \\
&=& \sum_{g=1}^{n-1} P_n(g) e^{itg} \ + \ o(1), \eea where the last equality follows from the definition of $P_n(j)$. We are changing variables in the double summation to exploit our knowledge of the average gap measure. Then
\bea
\lim_{n\to\infty} \E_m[\hnumnt] \ = \ \lim_{n\to\infty} \left(\sum_{g=0}^{n-1} P_n(g) g^t + o(1)\right) \ = \ \sum_{g=0}^\infty P(g) e^{itg} \ = \ \widehat{\nu}(t)
\eea (from the definition of $\nu(t)$ and $\hnut$), completing the proof. \end{proof}

\subsection{Variance of the Individual Gap Measures}\label{sec:varandfourth}

The last ingredient in our proof of Theorem \ref{thm:indgapmeasurebulk} is to show that the variance of the characteristic functions of the individual measures tends to zero. We give full details when our sequence is the Fibonaccis, and discuss the minor adjustments needed for the general case. We keep the argument as general as possible for as long as possible.

\begin{pro}\label{prop:variance}  Notation as above, we have \be \lim_{n\to\infty} {\rm Var}_n(t) \ := \ \lim_{n\to\infty} \E_m[(\hnumnt - \widehat{\nu_n}(t))^2] \ = \  0. \ee
\end{pro}

\begin{proof}
Let  \bea X_{j_1,j_1+g_1, j_2,j_2+g_2}(n)  & \ := \ & \#\Big\{m \in [G_n, G_{n+1}): G_{j_1}, G_{j_1+g_1}, G_{j_2}, G_{j_2+g_2}\ \text{in}\ m\text{'s\ decomposition,}\nonumber\\ & & \ \ \ \ \ \text{but\ not\ } G_{j_1+q}, G_{j_2+p}\ \text{for}\ 0<q<g_1, 0<p<g_2\Big\}; \eea if either $g_1$ or $g_2$ is zero then we count $m$'s with multiplicity equal to the number of gaps of length zero at $j_1$ or $j_2$.  Note \bea\label{variance}
\rm{Var}_n(t) \ := \ \E_m[(\hnumnt - \widehat{\nu_n}(t))^2] \ = \ \E_m[\hnumnt^2]-\widehat{\nu_n}(t)^2,\eea and we know $\lim_n \widehat{\nu_n}(t)^2$ from the proof of Proposition \ref{prop:expectedvalue}. We are therefore left with finding $\E_m[\hnumnt^2]$. As the algebra is a bit long in places (and there are a few technical obstructions which require careful book-keeping), we first quickly highlight the argument. As \be \widehat{\nu_n}(t) \ = \ \hnut + o(1), \ee by the triangle inequality it suffices to show $\lim_n \E_m[\hnumnt^2]$ converges to $\hnut^2$. For the limit of the average gap measure, the probability of a gap of length $g$ is $P(g)$, and is given by Theorem \ref{thm:avegapsbulk}. We have \bea \hnut^2 & \ = \ & \sum_{g_1=0}^\infty P(g_1) e^{itg_1} \sum_{g_2=0}^\infty P(g_2) e^{itg_2} \ = \  \sum_{g_1, g_2} P(g_1) P(g_2) e^{it(g_1 + g_2)}. \eea The goal is to show that $\lim_n \E_m[\hnumnt^2]$ differs from this by $o(1)$. We show they are close for each pair $(g_1,g_2)$, with the difference summable and $o(1)$ over all pairs. We are able to show the pairwise (almost) agreement by using our indicator variables $X_{j_1,j_1+g_1,j_2,j_2+g_2}$.\\ \ 

We now turn to the proof. In the calculation below $g_1$ and $g_2$ denote two arbitrary gaps that start at the two indices $j_1 \le j_2$; thus $g_1, g_2 \in \{0, 1, \dots, n-1\}$ and $j_1, j_2 \in \{1, 2, \dots, n\}$. As the number of indices in the proof is growing, we write $\ell_r(m)$ and $\ell_w(m)$ for the summands in $m$'s decomposition, making explicit the $m$ dependence. In the sum that follows, we have to separately deal with the case $r = w$. We have
\bea
\E_m[\hnumnt^2] & = & \frac{1}{G_{n+1}-G_n} \sum\limits_{m=G_n}^{G_{n+1}-1} \frac{1}{(k(m)-1)^2} \sum_{r=2}^{k(m)} e^{it(\ell_r(m)-\ell_{r-1}(m))} \sum_{w=2}^{k(m)} e^{it(\ell_w(m)-\ell_{w-1}(m))}  \nonumber\\
& = & \frac{1}{(G_{n+1}-G_n)(C_{{\rm Lek}}n+d)^2} \left(2\sum_{j_1 <j_2 \atop g_1,g_2}X_{j_1,j_1+g_1,j_2,j_2+g_2}(n) e^{itg_1} e^{itg_2}\right. \nonumber\\ & & \ \ \ \ \ +\ \left. \sum_{j_1,g_1} X_{j_1,j_1+g_1}(n) e^{2itg_1} \right) + o(1), \eea where the last line follows by using Remark \ref{rek:prodkmandapprox} to replace $1/(k(m)-1)^2$ with its average value up to a negligible error and then doing the same change of variables as before, and the factor of 2 is because we are taking $j_1 < j_2$. As the denominator is of order $n^2 (G_{n+1}-G_n)$ while $\sum_{j_1,g_1} X_{j_1,j_1+g_1}(n)$ is of order $n (G_{n+1}-G_n)$, the diagonal term does not contribute in the limit, and the factor of 2 vanishes when we sum over $j_1 < j_2$ (which gives $n^2/2 + O(n)$). Therefore
\bea
\E_m[\hnumnt^2] & = & \frac{2}{(G_{n+1}-G_n)(C_{{\rm Lek}}n+d)^2} \sum_{j_1<j_2\atop g_1,g_2}X_{j_1,j_1+g_1,j_2,j_2+g_2}(n) e^{it(g_1+ g_2)} \ + \ o(1)\nonumber\\
& = & \frac{2}{a_1\gl_1^{n}(\gl_1-1)(C_{{\rm Lek}}n+d)^2(1+o(1))} \Bigg(o(1)\nonumber\\ & &\ \ \ \ \ \ \ + \ \sum_{j_1<j_2} \sum_{g_1, g_2=0}^{n-1} X_{j_1,j_1+g_1,j_2,j_2+g_2}(n) e^{it(g_1 + g_2)} \Bigg). \label{lastlinevar1}
\eea There are several different cases to consider for the pair $(g_1,g_2)$: at least one of them could be 0, at least one of them could be 1, or both exceed 1. The argument is essentially the same in each case; the only difference comes from slight changes in how we count $X_{j_1,j_1+g_1,j_2,j_2+g_2}(n)$. Note that if we restricted ourselves to the Fibonacci numbers the first two cases cannot happen (if we consider only recurrences where all the coefficients are 0 or 1 then the first case cannot happen).


We first consider the case when $g_1 =g_2 = 1$. We chose to do this case in detail as it has some of the counting obstructions, and gives the general flavor. We determine $X_{j_1,j_1+1,j_2,j_2+1}(n)$ by counting the total number of decompositions in $[G_n, G_{n+1})$ which have a gap of length 1 from $G_{j_1}$ to $G_{j_1+1}$ (which we know how to do by Lemma \ref{lem:countlem2}) and then subtract the three different ways decompositions can have a gap of length 1 from $G_{j_1}$ to $G_{j_1+1}$ \emph{without} having a gap of length 1 at $G_{j_2}$ to $G_{j_2+1}$: (1) include $G_{j_1}, G_{j_1+1}$ and $G_{j_2+1}$ but do not include $G_{j_2}$; (2) include $G_{j_1}, G_{j_1+1}$ but do not include $G_{j_2}$ and $G_{j_2+1}$; and (3) include $G_{j_1}, G_{j_1+1}, G_{j_2}$ and $G_n$ but do not include and $G_{j_2+1}$. These three cases can be counted by Lemma \ref{lem:countlem1} and similar counting techniques.

Note it is sufficient to analyze these cases under the additional assumption that $j_2$ is at least $2L$ units from $j_1$ (where $L$ is the length of the recurrence). The reason is that the denominator has a factor of $n^2$; if $j_2$ is within a bounded distance of $j_1$ we only get an $n$ in the numerator, and the contribution is negligible.

There is one last technicality. If any of $j_1$, $j_2$, $j_2-j_1$, $n-j_1$ or $n-j_2$ is small then expanding a $G_\gamma$ (where $\gamma$ is one of these troubling indices) by the generalized Binet formula will not yield an error of size $o(1)$. This is the same issue we had in the proof of Theorem \ref{thm:avegapsbulk}, and is handled similarly. We introduce the notation $o_{j_1,j_2;n}(1)$, which is $o(1)$ if all of the combinations above are at least $\log^2 n$ away from 0, and bounded otherwise. Again the sum of this over all $j_1, j_2$ will be lower order. We therefore assume $j_2 \ge j_1 + 2L$. Because of the length of the lines, for formatting reasons we put the error term with the sum over all $j_1 < j_2$ and not over the restricted sums. We find
\bea \label{case1}
& & \sum_{j_1=1}^{n-1} \sum_{j_2=j_1+1}^{n-1} X_{j_1,j_1+1,j_2,j_2+1}(n) + o\left(n (G_{n+1}-G_n)\right)  \nonumber\\
& = & \sum_{j_1=1}^{n-2L} \sum_{j_2=j_1+2L}^{n}\Bigg[ (G_{n+1}-G_n)-G_{j_1+1}(G_{n-j_1}-G_{n-j_1-1})-G_{j_1}(G_{n-j_1+1}-2G_{n-j_1}+G_{n-j_1-1})  \nonumber\\
& &\ \ \ - \ (G_{n-j+1}-2G_{n-j}+G_{n-j_2-1})( G_{j_2}-G_{j_1+1}G_{j_2-j_1-1}-G_{j_1}(G_{j_2-j_1}-G_{j_2-j_1-1}))   \nonumber\\
& & \ \ \ - \ (G_{n-j}-G_{n-j_2-1})(G_{j_2}-G_{j_1+1}G_{j_2-j_1-1}-G_{j_1}(G_{j_2-j_1}-G_{j_2-j_1-1}))  \nonumber\\
& & \ \ \ - \ (G_{n-j}-G_{n-j_2-1})((G_{j+1}-G_{j})-G_{j_1+1}(G_{j_2-j_1}-G_{j_2-j_1-1}) \nonumber\\
& &\ \ \ - \ G_{j_1}(G_{j_2-j_1+1}-2G_{j_2-j_1}+G_{j_2-j_1-1}))\Bigg]  \nonumber\\
&=&  \sum_{j_1=1}^{n-2L} \sum_{j_2=j_1+2L}^{n} \Bigg[(a_1\lambda_1^n(\lambda_1-1)(1-a_1-a_1\lambda_1^{-1}(\lambda_1-1)(1+o(1)))  \nonumber\\
&&\ \ \ - \ a_1^2\lambda_1^{n-1}(\lambda_1-1)^2(1-a_1-a_1\lambda_1^{-1}(\lambda_1-1)(1+o_{j_1,j_2;n}(1)))   \nonumber\\
&&  \ \ \ - \ a_1^2\lambda_1^{n-1}(\lambda_1-1)(1-a_1-a_1\lambda_1^{-1}(\lambda_1-1)(1+o_{j_1,j_2;n}(1)))  \nonumber\\
&& \ \ \ - \ a_1^2\lambda_1^{n-1}(\lambda_1-1)^2(1-a_1-a_1\lambda_1^{-1}(\lambda_1-1)))(1+o_{j_1,j_2;n}(1)) \Bigg]\nonumber\\
& = &(1-a_1-a_1\lambda_1^{-1}(\lambda_1-1))(a_1\lambda_1^{n-1}(\lambda_1-1)(\lambda_1-a_1(\lambda_1-1)-a_1-a_1(\lambda_1-1))  \nonumber\\
& & \ \ \ \cdot\ (1+o(1)) \sum_{j_1=1}^{n-2L}\sum_{j_2=j_1+2L}^{n} 1 \nonumber \\
& = & \left(\frac{n^2 + O(n)}{2}\right) a_1 \lambda_1^n (\lambda_1-1) (1+o(1)) \label{varlastline2} ((\lambda_1(1-2a_1)+a_1)\lambda_1^{-1})^2.
\eea

Notice that as $n \rightarrow \infty$, (\ref{varlastline2}) times the coefficient in (\ref{lastlinevar1}) is, up to an error of size $o(1)$, $$\left(\frac{1}{C_{{\rm Lek}}}\lambda_1^{-1}(\lambda_1(1-2a_1)+a_1)\right)^2 \ = \ P(1)^2,$$ which cancels with corresponding piece in $\hnut^2$ in the difference $\E_m[\hnumnt^2]-\hnut$.

The other cases for $(g_1, g_2)$ can be handled similarly, and again we find that the contribution equals the corresponding terms from $\hnut^2$ in the difference $\E_m[\hnumnt^2]-\hnut$. The only complication is we need our error terms to be small enough so that we may sum over all pairs $(g_1, g_2)$. This is not a problem as our approach allows us to isolate the error term, which is small when summed over all pairs as the sum of $P_n(g)$ is bounded. Therefore, $\lim_{n\to\infty} {\rm Var}_n(t) = 0$, completing the proof.
\end{proof}



\subsection{Proof of Theorem \ref{thm:indgapmeasurebulk}}

We now turn to the proof of Theorem \ref{thm:indgapmeasurebulk}. We have already done the difficult part of the analysis in \S\ref{sec:avemomentind} and \S\ref{sec:varandfourth}. As the proof of convergence follows from standard probability arguments, we just sketch the details below.



\begin{proof}[Proof of Theorem \ref{sec:avemomentind}] To use L\'evy's continuity theorem (see \cite{FG}), we need a sequence of random variables $\{R_r\}$ (which do not have to be defined on the same probability space) whose characteristic functions $\{\varphi_r\}$ converge pointwise to the characteristic function $\varphi$ of a random variable $R$. If we have this, then the random variables $\{R_r\}$ converge in distribution to $R$ (i.e., the cumulative distribution functions of the $\{R_r\}$ converge to that of $\{R\}$ at all points of continuity).

For us, $R$ is essentially a geometric decay (it's a pure geometric decay for gaps of length 2 or more), and for each $n$ the $R_r$'s are the gap measures for each $m \in [G_n, G_{n+1})$. By our results on the convergence of the means $\E_m[\hnumnt]$ to $\hnut$ and the variance tending to zero, Chebyshev's inequality implies that given any $\epsilon > 0$, for each $n$ almost all $m$ have $\hnumnt$ within $\epsilon$ of $\hnut$ (notice we are able to do this for all $t$ simultaneously).

Our set $\{R_r\}$ is thus a collection of gap measures coming from $m \in [G_n, G_{n+1})$. We are able to take a subset of $m$ for each $n$ such that as $n\to\infty$ we have convergence of these measures to the average gap measure \emph{and} almost all $m \in [G_n, G_{n+1})$ are chosen. This completes the proof.
\end{proof}

\section{Longest Gap}\label{sec:longestgap}

\subsection{Overview}

We briefly describe our approach to determining the distribution and limiting behavior of the longest gap in Zeckendorf decompositions. The first step is to find a rational generating function $F(s,f)$, whose coefficients in $s$ give the number of decompositions with longest gap \texttt{less than} $f$.  This allows us to determine the cumulative distribution of the longest gap, which we expand by using a partial fraction decomposition. It is this last step where we need our additional restrictions on the roots of the associated polynomials $\calM(s)$ and $\calR(s)$. These lead to simpler partial fraction expansions, and minimizes the technical obstructions.

In the process of obtaining this exact expression, we need several technical lemmas about the behavior of the roots of the polynomials in the denominator of our generating functions $F(s,f)$.  In particular, in order to obtain estimates for the longest gap for large $n$, we use Rouch\'e's theorem, and show that the distribution is essentially determined by the behavior of a single root.  In turn, this root relates to the largest eigenvalue, $\lambda_1$, of the recurrence relation of the $G_i$'s. Approximating along these lines, we determine an asymptotic expression for the cumulative distribution function $P(n,f)$, which in the limit is doubly exponential:  $P(n,f) = \exp (C n \lambda_1^{-f}) + o(1)$; here the constant $C$ is a rational function of $\lambda_1$.

The error term in $P(n,f)$ is sufficiently small to allow us to determine asymptotic expressions for the mean and variance of the longest gap.  To do this, we sum over a sufficiently large interval $(\ell_n, h_n)$ containing the mean $\mu_n$, take partial sums, and then use the Euler-Maclaurin formula to smooth out our expression.  This yields a particularly nice asymptotic expression for the mean and variance of the longest gap.  This result is directly analogous to behavior seen in flipping coins (see \cite{Sch} and Remark \ref{rek:schilling}).

In order not to interrupt the flow of the arguments, we leave the proofs of the more technical lemmas and straightforward calculations to the appendices, while emphasizing the general approach of our argument.

\subsection{Exact Cumulative Distribution of the Longest Gap}

Our first step is to determine the cumulative distribution function of the longest gap. We begin by counting the number of $m \in [G_n, G_{n+1})$ with $L_n(m)$ less than some  $f \in \bbN$, and finding the associated generating function. As the longest gap grows on the order of $\log n$, it suffices to study $f \ge \log\log n$; in other words, in all arguments below we may assume $f$ is much larger than the length of the recurrence relation, and thus we do not need to worry about small numbers.


\begin{lem} \label{lem:genfunlem}
Let $f > j_L$. The number of $m \in [G_n, G_{n+1})$ with longest gap less than $f$ is given by the coefficient of $s^{n}$ in the generating function
\bea
F(s,f) & \ =\ & \frac{ 1-s^{j_L} }{\calM(s) + s^{f+1} \calR (s)},\eea where \bea\label{eq:defncalMcalR} \calM(s) & \ = \  & 1 - c_{1} s - c_{j_2 + 1} s^{j_2 + 1} - \dots - c_{j_L + 1} s^{j_L +1}\nonumber\\ \calR(s) & \ = \ & c_{1} + c_{j_2+1} s^{j_2} + \dots + (c_{j_L+1}-1) s^{j_L} \eea
and the $c_i$ and $j_i$ are defined as in \eqref{eq:recurrenceforlongestgaps}.
\end{lem}

Before beginning our proof, we fix some notation. Recall that our recurrence relation is written as \be\label{eq:recrelationforlongestgapinlongestgapsection} G_{n+1}\ =\ c_{j_1 +1} G_{n- j_1 } + c_{j_2+1} G_{n- j_2} + \dots + c_{j_L+1} G_{n-j_{L}}\ee with each $c_i$ non-zero and $j_1 = 0$.

\begin{itemize}

\item A \textbf{legal block of length $\ell$} is a sequence of non-negative integers $(a_i)_{i=1}^\ell$ where $a_i = c_i$ for $i \leq \ell$ and $a_\ell < c_\ell$.  Notice that a legal block $a$ must be of length $j_i$ for some $i$, in order to satisfy $a_i < c_{j_i}$.

\item A \textbf{string of zeroes of length $\ell$} is a sequence $(b_i)_{i=1}^\ell$ where each $b_i = 0$.

\item Denote the \textbf{concatenation of two sequences} $a = (a_i)_{i=1}^{\ell_a}$ and $b = (b_i)_{i=1}^{\ell_b}$ by $a\to b$, where \be a \to b\ =\ (u_i)_{i = 1}^{\ell_a+\ell_b},\ee with $u_i = a_i$ for $i \in [1,\ell_a]$ and $u_{i} = b_{i-s}$ for $i \in [\ell_a+1, \ell_a+\ell_b]$.

\item A \textbf{legal sequence} is a sequence of non-negative integers $({ \to}_{r = 1}^k \left( \eta_r \to z_r \right)) \to T$, where the $\eta_r$ are legal blocks, the $z_r$ are strings of zeroes, and $T$ is a \textbf{terminal block} (which is a sequence $(c_i)_{i=1}^\ell$ with $\ell < L$, with $L$ is the number of non-zero coefficients in the recurrence relation for the $G_i$'s; see \eqref{eq:recrelationforlongestgapinlongestgapsection}).  Informally, we say that a legal sequence consists of $k$ legal blocks, separated by strings of zeroes, and ended by a terminal block.  By definition, legal sequences of length $n$ are exactly those sequences that arise as decompositions of $x \in [G_n, G_{n+1})$.  We use $\left|a\right|$ to denote the length of a sequence $a$. 

\item Set \textbf{$T_f(s) = \calM(s) + s^f \calR(s)$}, with $\calM(s)$ and $\calR(s)$ as in \eqref{eq:defncalMcalR}. We denote its roots by $\alpha_{i;f}$, with $\alpha_{1;f}$ the smallest root. One of the difficulties in the analysis below is that these roots depend on $f$, though fortunately the only one that matters is $\alpha_{1;f}$, which exponentially converges to $1/\lambda_1$ (the reciprocal of the largest root of the characteristic polynomial of the recurrence relation for the $G_i$'s).

\end{itemize}

\begin{proof}[Proof of Lemma \ref{lem:genfunlem}]
By the Generalized Zeckendorf Theorem, Theorem \ref{thm:legaldecomp}, there exists a bijection between $m \in [G_n, G_{n+1})$ and legal decompositions of length $n$.  Accordingly, we count the number of length $n$ legal decomposition with longest gap less than $f$. As remarked above, we assume $f$ is at least $\log\log n$ for $n$ large, so in particular $f$ is much greater than the length of the recurrence.

A gap of length $g$ in the decomposition corresponds to a string of zeroes of length $g-1$ contained in the legal sequence. To count the number of decompositions with longest gap less than $f$,  we count the number of legal sequences of length $n$ with all strings of zeroes of length $ \leq f-2$.  First we consider legal blocks followed by a string of zeroes, or sequences of the form  $\eta \to z$ where $\eta$ is a legal block and $z$ is a string of zeroes.

There are $c_{j_i+1} - 1$ distinct legal blocks that have length $j_i+1$ and do not end in a zero. Let $\eta$ be a legal block that does not end in a zero. As $f > j_L$, the only sequences $\eta \to z$ with strings of zeroes  of length at least $f$ are those with $\left|z\right| \geq f$.  Let $N(r)$ be the number of length $r$ sequences  $\eta \to z$ that contain no string of zeroes of length $\geq f-1$. Since $\left|\eta\right| + \left|z\right| = r$,  we see that $N(r)$ is given by the generating function \beq \sum_{r=1}^\infty N(r) s^r \ = \   \left( (c_1 -1)s^{j_1+1}  + \dots + (c_{j_L+1} -1)s^{j_L + 1} \right)  \left( 1 + s + \dots + s^{f-1} \right). \eeq

For any $i \in \bbN$ such that $2 \leq i \leq  L$ there is exactly one legal block $\eta$ that has length $j_i+1$ and ends in a zero. There are no other legal blocks that end in a zero. Since the last non-zero term in $\eta$ is then $\eta_{j_{i-1}}$, the legal block contains a string of $g_{i-1} = j_{i} - j_{i-1}$ zeroes at the end.  Let $M(r)$ be the number of length $r$ sequences of legal blocks ending with a zero, followed by a string of zeroes, with no string of zeroes of length at least $f$; we denote this by $\eta \to z$.  As $f > j_L$, the longest string of zeroes of such a block has length $g_i + \left|z\right|$. So $\eta \to z$ contains no strings of zeroes of length at least $f$ if $\left|z\right| < f - g_i - 1$.  As $\left|\eta\right| + \left|z\right| = r$,  $M(r)$ is given by the generating function
\beq
\sum_{r=1}^\infty M(r)  s^{r} \ = \   s^{j_2 + 1} \left( \frac{1- s^{f - g_1}}{1-s} \right) + \dots +  s^{j_{L} +1} \left( \frac{1 - s^{f- g_{L-1}}}{1-s} \right).
\eeq

Finally, there is exactly one terminal block of length $r$ for each $r \geq 0$ and $r < j_L$. Thus the number $D(r)$ of length $r$ terminal blocks has the generating function $ \sum_{r=1}^\infty D(r) s^r \ = \      \frac{1-s^{j_L}}{1-s}$.

We now use these generating functions to find the number of legal sequences of length $n$ with $k$ legal blocks and all strings of zeroes of length less than $f$. Our decomposition based on the number of summands is similar to the analysis done in \cite{KKMW,MW1,MW2}; this is a natural way to split into cases, and provides a manageable route through the combinatorics. That is we fix $k$ and count the number of legal sequences $({ \to}_{i = 1}^k \left( \eta_i \to z_i \right)) \to T$ that do not contain a subsequence of zeroes of length at least $f$; recall the $\eta_i$ are legal blocks, $z_i$ are strings of zeroes, and $T$ is terminal.  Since the lengths of these separate components must sum to $n$,  the number of such length $n$ sequences is the coefficient of $s^n$ in \beq \left(\sum_{r=1}^\infty N(r)  s^{r} + \sum_{r=1}^\infty M(r)  s^{r}\right)^k  \sum_{r=1}^\infty D(r) s^r.  \eeq
To find $F(s,f)$, it remains only to sum the above expression over all $k$. Thus the generating function of the number of length $n$ legal sequences with longest gap $< f$ is
\bea
F(s,f) & \ = \ & \frac{1-s^{j_L}}{1-s} \sum_{k \geq 0} \bigg [ \left( (c_1 -1)s^{j_1+1}  + \dots + (c_{j_L+1} -1)s^{j_L +1} \right)  \left( \frac{1 - s^f}{1-s} \right) \nonumber\\ & & \ \ \
+\ \ s^{j_2+1} \left( \frac{1- s^{f - g_1}}{1-s} \right) + \dots +  s^{j_{L} +1} \left( \frac{1 - s^{f- g_{L-1}}}{1-s} \right) \bigg]^k.
\eea This is a geometric series, so we can evaluate our sum over $k$ and then use the relation $j_{i-1} + g_{i-1} = j_{i}$ to calculate the desired result.
\end{proof}

We have found a rational generating function for the cumulative distribution. To analyze it further, we first recall a standard lemma on partial fraction expansion.

\begin{lem}[Partial Fraction Expansion]
Let $R(s) = S(s)/T(s)$  be a rational function for $S, T \in \bbC[x]$ with $\deg(S) < \deg(T)$,  and assume $T$ has no multiple roots. Then the coefficient of $s^{n}$ in $R(s)$'s Taylor expansion around zero is
 \begin{equation}
  \sum_{i=1}^{\deg(T)} -\frac{S(\alpha_{i;f})}{\alpha_{i;f} T'(\alpha_{i;f})} \left( \frac 1 {\alpha_{i;f} } \right)^{n},
 \end{equation}  where $\{\alpha_{i;f}\}$ are the roots of $T$.
\end{lem}

Notice that in order to use this partial fraction expansion lemma, we need to ensure that the denominator of our generating function $F(s,f)$ has no multiple roots.  To achieve this, we impose some extra restrictions on our recurrence relation, and obtain the following.

\begin{lem} \label{lem:derivativebound}
Let $T_f(s) = \calM(s) + s^f\calR(s)$, where $\calM(s)$ and $\calR(s)$ have no multiple roots, and no roots of absolute value $1$. 
Then  there exists $\epsilon > 0$ and $F \in \bbN$ such that for all $f \geq F$ and all roots $\alpha$ of $T_f(s)$ we have $\left|T_f'(\alpha)\right| > \epsilon$.
 \end{lem}

We prove this lemma in Appendix \ref{sec:proofoflemderivativebound}, where we analyze the roots of the polynomial $T_f(s)$ as $f$ varies.  Essentially, the behavior of $T_f(s)$ is as we may expect; the roots of $T_f(s) = \calM(s) + s^f \calR(s)$ with absolute value less than one are close to the roots of $\calM(s)$ for large $f$, and the roots of $T_f(s)$ with absolute value greater than one are close to $\calR(s)$ for large $f$. We also see that the large number of roots of absolute value close to one will have little contribution.

Applying partial fractions, we immediately obtain the following expression for the cumulative distribution function, $P(n,f)$.

\begin{lem}\label{exactcdf}
Let $\{\alpha_{i;f}\}_{i = 1}^{f+j_L}$ be the roots of $T_f(s)$.  Then the cumulative distribution for the longest gap $P(n,f)$, the probability that a number $m \in [G_{n}, G_{n+1})$ has the longest gap in its Zeckendorf decomposition less than $f$ is
\beq P(n,f) \ = \   \frac{1}{G_{n+1} - G_n} \ \sum_{i=1}^{f+ j_L} \  \frac{1 - \alpha_{i;f} ^{j_L} }{ \alpha_{i;f} \ T_f'(\alpha_{i;f}) }
 \left( \frac{1}{\alpha_{i;f}} \right)^{n}.
  \eeq
\end{lem}

\subsection{Asymptotic Expansion for the CDF of the Longest Gap}

We need several facts about the roots of the polynomials $T_f(s)$ to use Lemma \ref{exactcdf}.   First, from the definition of $\calM(s)$, it is immediate that $\calM(s)$'s roots are exactly the inverse roots of the characteristic polynomial of the recurrence relation for the $G_i$.  We label the roots of this characteristic polynomial $\{\lambda_i\}_{i=1}^{j_L+1}$.   As we have noted in Binet's formula (Lemma \ref{lem:genbinetf}),  $\left|\lambda_1\right| >  \left| \lambda_2  \right|\geq  \cdots \geq \left|\lambda_{j_L+1}\right|$. Furthermore, we know that  $\lambda_1 \in \bbR$ and $\lambda_1 > 1$. In particular, this shows  that $\calM(s)$ has  a single smallest root $1 / \lambda_1$, which is real-valued and has absolute value less than $1$.  In turn, since  $T_f(s) = \calM(s) + s^f \calR(s)$, we may show that for large $f$, $T_f(s)$ has a smallest root that converges to $1/\lambda_1$.

\begin{prop}\label{prop:rootprop}
There exists $F \in \bbN$ and $R_{\max}, R_{\min} \in \bbR$ satisfying $ 1 / \lambda_1 < R_{\min} < \min(1, \left|1/\lambda_2\right|)$  such that for all $f \geq F$  every root $\alpha_{i,f}$ of $ \calM(s) + s^f \calR(s)$ has $\left|\alpha_{i,f}\right| < R_{\max}$, and such that the polynomial $\calM(s) + s^f \calR(s)$ has exactly one root \textbf{$\alpha_{1;f}$} with $\left|\alpha_{1;f}\right| < R_{\min}$.  Furthermore \beq
\alpha_{1;f} \ = \  \frac{1}{\lambda_1} +   \frac{\calM(\alpha_{1;f})}{\calG(\alpha_{1;f})} \alpha_{1;f}^f,
\eeq
where $\lambda_1$ is the largest eigenvalue of the recurrence relation for $G_i$ and $\calG(s) := -\calM(s) / (s-1/\lambda_1)$, a polynomial. Moreover, there exists $\delta > 0$ such that $ \left|\calG(\alpha_{1;f})\right|  > \delta$ for $f \geq F$.
\end{prop}

The proof of Proposition \ref{prop:rootprop} is standard, and given in Appendix \ref{sec:proofofproprootprop}.

The roots $\{\alpha_{i;f}\}_{i = 1}^f$ appear in the terms of the sum in Lemma \ref{exactcdf} as $\alpha_{i;f}^{-n}$, and the smallest root $\alpha_{1;f}$ dominates the sum.  Being careful to deal with coefficients, and approximating $\alpha_{1;f} ~\lambda_1$, we obtain our claimed asymptotic expression for the cumulative distribution function of the longest gap.

\begin{proof}[Proof of Theorem \ref{thm:longestgap}(1)]
From Lemma \ref{exactcdf}, we have
\beq P(n,f) \ = \   \frac{1}{G_{n+1} - G_n} \ \sum_{i=1}^{f+ j_L} \  \frac{1 - \alpha_{i;f} ^{j_L} }{ \alpha_{i;f} \ T_f'(\alpha_{i;f}) }
 \left( \frac{1}{\alpha_{i;f}} \right)^{n}.
 \eeq

By definition, we have that $1/\lambda_1 < R_{\min} < \left| 1/\lambda_2 \right|$. Therefore, by the generalized Binet formula (Lemma \ref{lem:genbinetf}) \be G_{n+1} - G_n \ = \ C' \lambda_1^n + O\left((1/R_{\min})^{n}\right)\ee for some $C' \in \bbR$.  Further, for any root $\alpha_{i;f} \neq \alpha_{1;f}$ we have $\left|\alpha_{i;f}\right| > R_{\min}$. Also, by Lemma \ref{lem:derivativebound} there is a bound $B \in \bbR$ such that $\left| 1/ T_f'(\alpha_{i;f}) \right| < B$ for all roots $\alpha_{i;f}$ and for all $f \geq F$.

 We see that for $\alpha_{1;f}$, the critical root from before, that  \beq
 P(n,f) \ = \  - \frac{(1-\alpha_{1;f}^{j_L})}{C \alpha_1T_f'(\alpha_{1;f})} \left( \frac {1} {\lambda_1 \alpha_{1;f} }\right) ^{n}  + O\left( f \left(  {\lambda_1 R_{\min }} \right)^{-n}   \right);
 \eeq note $\lambda_1 R_{{\rm min}} > 1$.

Next we use the relation
 \beq\label{root asymp}
  \alpha_{1;f} \ = \  \frac{1}{ \lambda_1} + \frac{\alpha_{1;f}^f \calR(\alpha_{1;f})}{\calG(\alpha_{1;f})}
  \eeq
  from Proposition \ref{prop:rootprop} to express our formula in terms of $\lambda_1$. \  Accordingly let $C = \lambda_1/C'$. Since $\alpha_{1;f}< R_{\min} < 1$ and $\calG(\alpha_{1;f})$ is bounded away from zero, this shows that $\alpha_{1;f}$ converges to $1/\lambda_1$ exponentially fast. Substituting $\eqref{root asymp}$ three times and recalling $T_f(s) = \calM(s) + s^f \calR(s)$ gives
  \beq
 P(n,f) \ = \  \frac{ - C\left(1- {\lambda_1^{j_L}}\right) + O(\alpha_{1;f}^f)} {\left[\frac{1}{\lambda_1} \calM'(\frac{1}{\lambda_1}) + O(\alpha_{1;f} ^{f}) \right] +  O(f \alpha_{1;f}^f)} \left( 1 + \lambda_1 \frac{\alpha_{1;f}^f \calR(\alpha_{1;f})}{\calG(\alpha_{1;f})}  \right) ^{-n}  + O\left( f \left(  {\lambda_1 R_{\min }} \right)^{-n}   \right).
 \eeq
As  $\alpha_{1;f} < R_{\min} < 1$ and $ \left| 1 + \lambda_1 \frac{\alpha_{1;f}^f \calR(\alpha_{1;f})}{\calG(\alpha_{1;f})}\right|^{-n} < 1$, we obtain the asymptotic expression
 \beq
 P(n,f) \ = \   - \frac{C(1-\lambda_1^{j_L})} {\frac{1}{\lambda_1} \calM'(\frac{1}{\lambda_1})} \left( 1 + \lambda_1 \frac{\alpha_{1;f}^f \calR(\alpha_{1;f})}{\calG(\alpha_{1;f})}  \right) ^{-n}
 + O\left( f R_{\min}^f   \right)
  + O\left( f \left(  {\lambda_1 R_{\min }} \right)^{-n}   \right).
 \eeq
Exponentiating gives
 \beq
 \left( 1 + \lambda_1 \frac{\alpha_{1;f}^f \calR(\alpha_{1;f})}{\calG(\alpha_{1;f})}  \right) ^{-n} \ = \  \exp\left( - n \log \left(1 + \lambda_1 \frac{\alpha_{1;f}^f \calR(\alpha_{1;f})}{\calG(\alpha_{1;f})}  \right)\right).
 \eeq
Using the Taylor expansion of $\log(1 + s)$, we see that
 \beq
  \left( 1 + \lambda_1 \frac{\alpha_{1;f}^f \calR(\alpha_{1;f})}{\calG(\alpha_{1;f})}  \right) ^{-n} \ = \  \exp\left( - n \lambda_1 \frac{\alpha_{1;f}^f \calR(\alpha_{1;f})}{\calG(\alpha_{1;f})} \right)  \exp \left(O(n \alpha_{1;f}^{2f} )\right).
   \eeq
Taylor expanding again, and using $\exp(-u) = 1 + O(u)$ for $\left|u\right| < 1$, yields our penultimate asymptotic expansion:
\beq
P(n,f) \ = \   - \frac{C(1-\lambda_1^{j_L})} {\frac{1}{\lambda_1} \calM'(\frac{1}{\lambda_1})}
 \exp\left( - n \lambda_1 \frac{\alpha_{1;f}^f \calR(\alpha_{1;f})}{\calG(\alpha_{1;f})} \right)
+ O\left(n \alpha_{1;f} ^{2f}\right)
+O\left( f R_{\min}^f   \right)
  + O\left( f \left(  {\lambda_1 R_{\min }} \right)^{-n}   \right).
 \eeq
By $\eqref{root asymp}$ and since $\alpha_{1;f} < R_{\min} < 1$ we have the relationship \beq\label{eq:alpha1ferror}  \alpha_{1;f} ^{f} \ = \  \left( \frac{1}{\lambda_1} \right)^{f} \left( 1 + \lambda_1 \frac{\alpha_{1;f}^f \calR(\alpha_{1;f})}{\calG(\alpha_{1;f})}  \right)^{f} \ = \  \left(\frac{1}{\lambda_1} \right)^{f} + O\left(f \left(\frac{ R_{\min}}{\lambda_1} \right)^f \right),\eeq
and thus
 \beq
 \exp\left( - n \lambda_1 \frac{\alpha_{1;f}^f \calR(\alpha_{1;f})}{\calG(\alpha_{1;f})} \right) \ = \
 \exp\left( - n \lambda_1 \frac{ \calR(\frac{1}{\lambda_1})}{\lambda_1^f \calG(\frac{1}{\lambda_1})} \right) + O\left(n f \left(\frac{ R_{\min}}{\lambda_1} \right) ^f \right).
 \eeq
Squaring \eqref{eq:alpha1ferror} gives
 \beq
 O \left( n \left( \alpha_{1;f} \right)^{2f}\right) \ = \  O\left(n \left( \frac{1}{\lambda_1} \right)^{2f} \right) + \ O\left(nf \frac{ R_{\min}^f}{\lambda_1^{2f}}\right),
 \eeq and thus
 \bea \nonumber
P(n,f) \ = \   - \frac{C(1-\lambda_1^{j_L})} {\frac{1}{\lambda_1} \calM'(\frac{1}{\lambda_1})}
  \exp\left( - n \lambda_1^{-(f-1)} \frac{ \calR(\frac{1}{\lambda_1})}{ \calG(\frac{1}{\lambda_1})} \right)
+ O\left(n f \left(\frac{ R_{\min}}{\lambda_1} \right) ^f \right) \\
+ O\left(n \left( \frac{1}{\lambda_1} \right)^{2f} \right)
 + O\left( f \left(  {\lambda_1 R_{\min }} \right)^{-n}   \right). \label{thiseqn}
 \eea
Further, since we always have $P(n,n+1) = 1$, substituting $f = n$ into \eqref{thiseqn} gives \be\label{eq:lasteqinproofandshouldcheck} \lim_{n \to \infty} P(n,n+1) \ =\  - \frac{ C\lambda_1 (1-\lambda_1^{j_L})} {\calM'(\frac{1}{\lambda_1})}.\ee It follows that  $ - C \lambda_1(1-\lambda_1^{j_L}) / \calM'(\frac{1}{\lambda_1}) = 1$, completing the proof of Theorem \ref{thm:longestgap}(1).
\end{proof}

\subsection{Mean and Variance of the Longest Gap }\label{sec:meanvarlongestgap}

We use our asymptotic expression for the cumulative distribution function to calculate statistics of the longest gap distribution. \emph{Remember that our cumulative distribution is defined for the longest gap being \texttt{less than} a given value.} Thus in the analysis below it is a little easier to first find, not the mean and variance of the random variable $X$ denoting the longest gap, but the mean and the variance of the random variable $Y$ which is one more than the longest gap. As $X = Y-1$, to obtain the mean and the variance of $X$ just requires subtracting 1 from the mean of $Y$ (the variance is unchanged).

The mean $\mu_{n;Y}$ and the variance $\sigma_n^2$ of $Y$ are given by \beq \label{meanvar} \mu_{n;Y} \ = \   \sum_{g=1}^n g \left( P(n, g) - P(n,g-1) \right);  \ \ \ \ \  \sigma_n^2 \ = \  \sum_{g=1}^n  g^2 \left( P(n, g) - P(n,g-1) \right) - \mu_{n;Y}^2. \eeq Thus our desired mean (for the longest gap) is $\mu_n = \mu_{n;Y} - 1$, and the variance is $\sigma_n^2$.

As our asymptotic expression for $P(n,g)$ is only accurate for values of $g$ on the order of $\log n$ or larger, we replace the sums in \eqref{meanvar} from $1$ to $n$ by sums from $\ell_n$ to $h_n$, for suitable choices of $\ell_n$ and $h_n$, so that the error from restricting the summation is negligible.  This is possible due to the very tight double exponential behavior, which we proved in Theorem \ref{thm:longestgap}(1).  In particular, we have the following proposition.

\begin{prop} \label{restrictsum}
Choosing $c, C \in \R$ such that $0 < c < 1/\lambda_1  $ and $C > \max(6, 4\log \lambda_1)$, we let $\ell_n =  \lfloor c \log(nK) \rfloor$ and $h_n = \lfloor C \log(nK) \rfloor$ (remember $\lambda_1 > 1$ and $K$ is as in Theorem \ref{thm:longestgap}(2)). We find that
\bea  \mu_{n;Y} & \ = \  & \sum_{g=\ell_n}^{h_n} g \left( P(n, g) - P(n,g-1) \right) + o(1) \nonumber\\
\sigma_n^2 & \ = \ & -\left(\mu_{n;Y}^2 - \sum_{g=\ell_n}^{h_n}  g^2 \left( P(n, g) - P(n,g-1) \right)\right) + o(1).\eea
\end{prop}

With these values of $h_n$ and $\ell_n$, to prove the above proposition only requires the crudest bounds; we do this in Appendix \ref{sec:appendixrestrictingsums}. We now finish the proof of our main theorem on longest gaps.

\begin{proof}[Proof of Theorem \ref{thm:longestgap}(2)]
We work simultaneously with $\mu_{n;Y}$ and $\sigma_n^2$.  In preparation for approximating our sums with integrals, we first sum by parts so that
\bea
\mu_{n;Y} & \ = \ &
\ (h_n+1) P(n, h_n) - \ell_n P(n, \ell_n -1) \ - \ \sum_{g= \ell_n}^{h_n}P(n, g) \ + o(1)\nonumber\\
\sigma_n^2 & \ = \ &  -\left(\mu_{n;Y}^2 - \ (h_n+1)^2 P(n, h_n) + \ell_n^2 \ P(n, \ell_n-1) \ +  \sum_{g= \ell_n}^{h_n} (2g +1) P(n, g)\right)  + o(1).\nonumber\\
\eea
From Theorem \ref{thm:longestgap}(1), we know that $\ell_n^2 P(n,\ell_n) \to 0$ and $h_n^2 P(n, h_n) \to h_n^2$ for large $n$, and hence
\bea
\mu_{n;Y} & \ = \  & \ (h_n+1)  - \ S_1 \ + \ o(1)\nonumber\\
\sigma_n^2 & \ = \ &  -(\mu_{n;Y}^2 - (h_n+1)^2 +  2 S_2 +   S_1)  + o(1)
\eea
for \be S_1\ :=\ \sum_{g=\ell_n}^{h_n}P(n, g), \ \ \ \ \ S_2\ :=\ \sum_{g=\ell_n}^{h_n}g P(n, g).\ee With $K =  \lambda_1 \calR(1/\lambda_1) / \calG(1/\lambda_1)$, our estimates from Theorem \ref{thm:longestgap} give us
\bea
S_1 & \ = \ &  \sum_{g= \ell_n}^{h_n}  \exp \left( - n K \lambda_1^{- g}  \right) + O\left(  n^{-\delta} \log n \right) \nonumber\\
S_2 & \ = \ & \sum_{g= \ell_n}^{h_n}  g \exp \left( - n K \lambda_1^{-g} \right) + O\left(  n^{-\delta} (\log n)^2 \right).
\eea

Now we apply the Euler-Maclaurin formula to $S_1$ and $S_2$, and find
\bea \label{i1}
S_1 \ = \  \int_{\ell_n}^{h_n} \exp \left( - n K \lambda_1^{-t} \right) \ dt  \ +\  \frac{1}{2}P(n,t)  \bigg |_{t = \ell_n}^{h_n}  \  + \ \mbox{Error}^1_{EM} + o(1)
\eea
and
\bea \label{i2}
S_2 \ = \  \int_{\ell_n}^{h_n} t \ \exp \left( - n K \lambda_1^{-t} \right)  \ dt  \ +\  \frac {1} {2} t \  P(n,t) \bigg |_{t = \ell_n}^{h_n}  \  + \ \mbox{Error}^2_{EM} + o(1).
\eea In Appendix \ref{sec:appendixeulermaclaurin} we show that $\mbox{Error}^1_{EM}=o(1)$ and $\mbox{Error}^2_{EM}=1+o(1)$, and thus the two errors above are negligible. The boundary terms approach $1/2$ and $h_n/2$, respectively, since $P(n,h_n) \to 1$ while $P(n,\ell_n) \to 0$ so fast that $\ell_n P(n,\ell_n) \to 0$. We are left with analyzing the two integrals.

Define $w(t) =   \exp\left( - t \log \lambda_1   + { \log \left(  n K  \right ) }  \right)  $, with $w'(t) = - w(t) \log \lambda_1$ and $t =  \frac{ \log (nK) - \log w}{\log \lambda_1}$.  Writing $I_1$ for the integral in \eqref{i1} and $I_2$ for the integral in \eqref{i2}, integrating by parts yields
  \bea
   I_1 & \ = \ & t \ e^{-w(t)} \bigg |_{\ell_n}^{h_n} \ +\ \int_{\ell_n}^{h_n} \  \ t \   e^{-w(t)}  \ w'(t) \   dt \nonumber\\
   I_2 & \ = \ &  \frac{t^2}{2} \ e^{-w(t)} \bigg |_{\ell_n}^{h_n} \ +\ \int_{\ell_n}^{h_n} \  \ \frac{t^2}{2} \   e^{-w(t)}  \ w'(t) \   dt.
  \eea
Letting $u = w(t)$ gives
  \bea
    I_1 & \ = \ & t \ e^{-w(t)} \bigg |_{\ell_n}^{h_n} \ +\ \int_{w(\ell_n)}^{w(h_n)} \  \frac{ \log (nK) - \log u}{\log \lambda_1}  \   e^{-u}  \  \   du, \nonumber\\
 I_2 & \ = \ & \frac{1}{2} \left( t^2 \ e^{-w(t)} \bigg |_{\ell_n}^{h_n} \ +\ \int_{w(\ell_n)}^{w(h_n)} \  \left(\frac{ \log (nK) - \log u}{\log \lambda_1}\right)^2 \   e^{-u}   \   du \right).
 \eea
We then expand the integrals and note that by our choices of $h_n$ and $w_n$ we have that $w(h_n) = 0 + o(1)$ and $w(\ell_n)$ is positive and tends to infinity with $n$.  Then, using the well known identities (see 4.331.1 and 4.335.1 of \cite{GR})
\beq
 \int_{0}^{\infty}  \log \left( u \right)  e^{-u} \ du \  =  \ - \gamma, \ \ \ \  \int_{0}^{\infty} \left( \log  u \right)^2  e^{-u} \ du \  =   \ \gamma^2 + \frac{\pi^2}{6} \eeq
with $\gamma$ the Euler-Mascheroni constant (note on page xxxii of \cite{GR} they set $C = \gamma$), we may evaluate our integrals to obtain
\bea
  I_1 & \ = \ &  t \ e^{-w(t)} \bigg |_{\ell_n}^{h_n} \ - \ \frac{\log(nK)  + \gamma }{{\log \lambda_1}} + o(1) \nonumber\\
 I_2 & \ = \ & \frac{1}{2} \left(  t^2 \ e^{-w(t)} \bigg |_{\ell_n}^{h_n} \ - \ \frac{1}{{(\log \lambda_1})^2} \ \left( \log(nK)^2 + 2 \gamma \log(nK) + \gamma^2 + \frac{\pi^2}{6}  \right) \right) + o(1).
  \eea

Our claimed values for the mean and variance now follow by evaluating the above and substituting. We note that $t^r e^{-w(t)} \big|_{\ell_n}^{h_n} = t^r P(n,t) \big |_{\ell_n}^{h_n} = h_n^r + o(1)$ (for $r \in \{1, 2\}$), and $t P(n,t) \big |_{\ell_n}^{h_n} = h_n + o(1)$ for our choices of $h_n$  and $\ell_n$. For  example, the mean is \be \mu_{n;Y} \ = \ (h_n+1) - \left(h_n - \frac{\log(nK) + \gamma}{\log \lambda_1} + \frac12\right) + o(1); \ee as $\mu_n = \mu_{n;Y}-1$ we immediately find \be \mu_n \ = \ \frac{\log(nK) + \gamma}{\log \lambda_1} - \frac12 + o(1). \ee
\end{proof}

\begin{rek} In the analysis above we took $h_n = \lfloor C \log ( nK) \rfloor$  with $C > \max(6, 4\log \lambda_1)$. As we saw from the subtraction, the constant here can be replaced with any sufficiently large value; however, we need $h_n$ to be at least this large to facilitate the error analysis in the appendix arising from the truncation of the sums. \end{rek}

\section{Concluding Remarks}\label{sec:conclusion}

Building on the combinatorial vantage introduced in \cite{KKMW} and its sequels, we are able to determine the limiting behavior for the distribution of gaps in the bulk, both on average and almost surely for the individual gap measures, as well as mean and variance of the longest gap. A natural future project is to remove some of the assumptions we have made on the recurrence relation. We expect the answers in these cases to be essentially the same, but the resulting algebra will be more involved.

An additional line of investigation is to apply these methods to other decompositions, for example the $f$-decompositions introduced in \cite{DDKMMV}.

\begin{defi} Given a function $f:\N_0\to\N_0$ and a sequence of integers $\{a_n\}$, a sum $m = \sum_{i=0}^k a_{n_i}$ of terms of $\{a_n\}$ is an \emph{$f$-decomposition of $m$ using $\{a_n\}$} if for every $a_{n_i}$ in the $f$-decomposition, the previous $f(n_i)$ terms ($a_{n_i-f(n_i)}$, $a_{n_i-f(n_i)+1}$, $\dots$, $a_{n_i-1}$) are not in the $f$-decomposition. \end{defi}

To see that this generalizes the standard Zeckendorf decomposition, simply take $a_n$ to be the $n$\textsuperscript{th} Fibonacci number and $f(n) = 1$ for all $n$. The authors prove that for any $f:\N_0\to\N_0$ there exists a unique sequence of natural numbers $\{a_n\}$ such that every positive integer has a unique legal $f$-decomposition in $\{a_n\}$. Interestingly, certain choices of $f$ lead to sequences defined by a recurrence relation with \emph{negative} coefficients in a fundamental way. This means there is no equivalent definition using only non-negative coefficients (for example, the Fibonaccis can be defined by $F_{n+1} = 2F_n - F_{n-2}$, but they are also given by the more standard relation $F_{n+1} = F_n + F_{n-1}$). One example is their $b$-bin decompositions. We break the natural numbers into bins of length $b$, and say a decomposition is legal if we never choose two elements from the same bin, nor two adjacent elements from two consecutive bins. This leads to a periodic formula for the associated $f$. For example, if $b=3$ our sequence of $a_n$'s starts 1, 2, 3, 4, 7, 11, 15, 26, 41, 56, 97, 153, and satisfies the recurrence $a_n = 4a_{n-3}-a_{n-6}$, while if $b=2$ we recover the standard Zeckendorf decomposition involving Fibonacci numbers.

\appendix

\section{Results on Roots of Associated Polynomials}


\subsection{Proof of Proposition \ref{prop:rootprop}}\label{sec:proofofproprootprop}

\begin{proof}[Proof of Proposition \ref{prop:rootprop}]

We apply Rouch\'e's theorem to obtain the appropriate bounds on roots. Choose $R_{\max} > \max(1, \beta_1, \dots, \beta_{j_L})$ where $\{\beta_i\}$ are the roots of $\calR(s)$.  Let $m' > 0$ be the minimum absolute value of $\calR(s)$ on the circle of radius $R_{\max}$, and let $M'$ be the maximum absolute value of $\calM(s)$. We choose $C' > \frac{\log (M'/m')}{ \log R_{\max}}$; for $f \geq C'$ and all $x$ on the circle of radius $R$ we obtain \beq | s^f  \calR(s) | \ \geq\ (R_{\max} ) ^{C'}  m'\ > \ M' \ \geq\  |\calM(s)|.\eeq By Rouch\'e's Theorem on the disk of radius $R_{\max}$, we see that $\calM(s)$ and $\calM(s) + s^f \calR(s)$ have the same number of roots within this disk; that is, $\calM(s) + s^f \calR(s)$ must have all its roots within the disk of radius $R_{\max}$.

Next choose any $R_{\min}$ such that  $1/ \lambda_1 < R_{\min} < \min_i (1/ \lambda_i)$ and $R_{\min} < 1$. Suppose $\calR(s)$ has a maximum absolute value $M$, and $\calM(s)$ has minimum absolute value $m$ on the circle of radius $R_{\min}$. We know $m > 0$ since $\calM(s)$ has no zeroes of absolute value $R_{\min}$.

Now let $C >  \frac{\log (m/M)}{ \log R_{\min}} $. Then for $f\geq C$,  \beq |s^f\calR(s)|\ \leq\ |s^fM|\ \leq\ R_{\min}^CM \ = \ m<M(s).\eeq  Using Rouch\'e's Theorem on the disk of radius $R_{\min}$, we see that $\calM(s)$ and $\calM(s) + s^f \calR(s)$ have the same number of roots within this disk; that is, $\calM(s) + s^f \calR(s)$ has exactly one root \textbf{$\alpha_{1;f}$} with $|\alpha_{1;f}| < R_{\min}$.  Taking  $F = \max(C', C)$ yields the desired result.

Next factor $\calM(s)$ as $\calM(s) =  -(s - 1/\lambda_1) \calG(s)$.  Then the root $\alpha_{1;f}$ satisfies $-(\alpha_{1;f} - 1/\lambda_1) \calG(s) + \alpha_{1;f}^f \calR(\alpha_{1;f}) = 0$, so  \beq \alpha_{1;f} \ = \  \frac{1}{ \lambda_1} + \frac{\alpha_{1;f}^f \calR(\alpha_{1;f})}{\calG(\alpha_{1;f})}. \eeq As $0 < \alpha_{1;f} < R_{\min} < 1$ for all $f \geq F$, note that $\calG(\alpha_{1;f})$ has roots with absolute value strictly greater than $R_{\min}$. It follows that $\calG(\alpha_{1;f}) > \delta$ for some $\delta > 0$ and for all $f \geq F$.
 \end{proof}




\subsection{Proof of Lemma \ref{lem:derivativebound}}\label{sec:proofoflemderivativebound}

\begin{proof}[Proof of Lemma \ref{lem:derivativebound}]
Fix $\epsilon > 0$. By continuity (and compactness of the circle) there exist $a, \eta > 0$ such that for all $s \in \bbC$ with $1-a< |s| <1+a$ we have $|\calR(s)|,|\calM(s)| > \eta$.   
Notice  $T_f(\alpha) = \calM(\alpha) +  \alpha^f \calR(\alpha) = 0$, and that $\calM(s) = 1 - \calR(s) s - s^{j_L + 1}$. These relations show that for any root $\alpha$ of $T_f$ that $\calR(\alpha) \neq 0$, since otherwise this would imply $\calM(\alpha) = 0$ and so $1 - \alpha^{j_L + 1} = 0$ would show that $| \alpha | = 1$, contradicting our hypothesis.

So $ - \calM(\alpha) / \calR(\alpha)  = \alpha^{f}$ and we have that $T'_f(\alpha) = \calM'(\alpha) + f \alpha^{f-1} \calR(\alpha) + \alpha^f \calR'(\alpha) = \calM'(\alpha) + f  \calM(\alpha)/\alpha + \alpha^f \calR'(\alpha)$. Since $| \alpha | < R_{\max}$ (see  Proposition \ref{prop:rootprop}) we have that $\calM(\alpha), \calM'(\alpha), \calR(\alpha)$, and $\calR'(\alpha)$ are bounded independently of $f$ by $B > 0$.

By our conditions on $\calM(s)$ and $\calR(s)$, we may choose $r, \delta > 0$ such that within $r$ of each root of $\calM(s)$ and $\calR(s)$, we have $|\calR'(s)|, |\calM'(s)| > \delta$.

For roots $\alpha$ of $T_f(s)$ such that $|\alpha| > 1+a$, choose $F$ large enough that $(1+a)^F > B \delta F + B + \epsilon $.  Then for $f > F$ we have $|\alpha^f \calR(\alpha)| > |\calM(\alpha) + f \alpha \calM(\alpha)| + \epsilon$, so $T'_f(\alpha) > \epsilon$.  For $| \alpha | < 1-a$ choose $F$ large enough that $F \delta R_{{\rm min}} > (1+a)^F B + B + \epsilon$ since.  Then $T'_f(\alpha) > \epsilon$, and if $ 1-a < | \alpha | < 1+a$, then $\eta/B < |\alpha^f| < B/\eta$, since  $ - \calM(\alpha) / \calR(\alpha)  = \alpha^{f}$.  Thus we may choose $F$ so that for all $f > F$,  we have $f \eta /(1+a) > B^2/\eta + B + \epsilon$, so that $|\calM(\alpha)/\alpha | >  | \calM '(\alpha) + \alpha^f \calR'(\alpha)| + \epsilon$.  Taking the maximum of these choices, we may always choose $F$ large enough that $T'_f(\alpha) > \epsilon$ for all necessary $\alpha$. \end{proof}

\section{Restricting Summations}\label{sec:appendixrestrictingsums}

Before we proceed further, we introduce some notation. First define $K = \frac{\lambda_1 \calR(\frac{1}{\lambda_1})}{\calG(\frac{1}{\lambda_1})} $ so that
\beq\label{eq:appendixrestrictingsumdefnK}
\frac{\log (nK)}{\log \lambda_1} \ = \  \frac{ \log \left(\frac{\lambda_1 \calR(\frac{1}{\lambda_1})}{\calG(\frac{1}{\lambda_1})}n  \right ) }{ \log \lambda_1 }. \eeq
Next, choose $c, C$ in $\bbR$ such that $0 < c < \frac{1}{\log \lambda_1}$ and $C > \max(6, 4\log \lambda_1)$. We denote by $\ell_n$ (for low) the quantity $\lfloor c \log (nK) \rfloor$ and by $h_n$ (for high) the quantity $\lfloor C \log (nK) \rfloor$. A consequence of our choice of $c$ is that for all choices of $f$ bounded below by $\ell_n$, there exists a $\delta > 0$ so that the error terms in Theorem \ref{thm:longestgap}(1) are $O(n^{-\delta})$.


\begin{proof}[Proof of Proposition \ref{restrictsum}] Since $P(n,g)$ is monotonically increasing, we have that
\be\sum_{g=1}^{\ell_n}  g^2 \left( P(n,g) - P(n,g-1) \right) \ \ll \  \ell_n^3 \ P(n, \ell_n+1);\ee note if we can bound this sum by $o(1)$ then a similar analysis works when we have $g$ instead of $g^2$ on the left. Thus
 \beq \sum_{g=1}^{\ell_n }  g^2 \left( P(n,g) - P(n,g-1) \right) \ \ll\  \ell_n^3 \left[  \exp\left( - n \lambda_1^{-\ell_n} /K \right)   +  \ O\left(n^{-\delta} \right)  \right],
\eeq
which implies
\be \sum_{g=1}^{\ell_n }  g^2 \left( P(n,g) - P(n,g-1) \right)  \ \ll\  (\log n)^{3} e^{-K^2 n^{1-c \log \lambda_1}} + \ O\left( (\log n)^3 n^{-\delta}\right).\ee  As $c \log \lambda_1 < 1 $, the left hand sum tends to zero in the limit.

Similarly, \be \sum_{g= h_n }^{n}  g^2 \left( P(n,g) - P(n,g-1) \right) \ \ll \  n^3 \left[ 1 - P(n,h_n+1) \right];\ee again it suffices to show this sum is $o(1)$ to show the related sum (with $g^2$ replaced by $g$) is $o(1)$. Therefore
\beq
\sum_{g= h_n  }^{n}  g^2 \left( P(n,g) - P(n,g-1) \right) \ = \  n^3 \left[ 1 - \exp\left( - n \lambda_1^{-(h_n-1)} \frac{ \calR(\frac{1}{\lambda_1})}{ \calG(\frac{1}{\lambda_1})} \right)  + O\left(n h_n \left(\frac{ R_{\min}}{\lambda_1} \right) ^{h_n} \right) \right].
\eeq After Taylor expanding the exponential, we bound the left hand sum with \be \sum_{g= h_n  }^{n}  g^2 \left( P(n,g) - P(n,g-1) \right)\  \ll \ O\left(n^3 \left( n^{1-6} + n (\log n) \ n^{-5} \right) \right)\ \ll\ o(1).\ee We have therefore shown
\be \lim_{n \to \infty} \left| \mu_{n;Y}  - \sum_{g= \ell_n  }^{h_n} g \left( P(n,g) - P(n,g-1) \right) \right| \ = \  0\ee
and
\be \lim_{n \to \infty} \left| \sigma_n^2  - \mu_{n;Y} - \sum_{g= \ell_n  }^{h_n} g^2 \left( P(n,g) - P(n,g-1) \right) \right| \ = \  0\ee as desired.
\end{proof}


\section{Error terms in the Euler-Maclaurin Formula}\label{sec:appendixeulermaclaurin}

Recall that we wish to estimate $\sum_{g= \ell_n}^{h_n}  \exp \left( - n K \lambda_1^{- g}  \right)$ using the Euler-Maclaurin formula. In \S\ref{sec:meanvarlongestgap} we showed that this sum equals
\be \int_{\ell_n}^{h_n} \exp \left( - n K \lambda_1^{-t} \right) \ dt  \ +\  \frac{1}{2}P(n,t)  \bigg |_{t = \ell_n}^{h_n}  \  + \ \mbox{Error}^1_{EM},\ee where $K$ is defined in \eqref{eq:appendixrestrictingsumdefnK}. To complete our determination of the mean, we must bound $\mbox{Error}^1_{EM}$, which is the error generated by the usage of Euler-Maclaurin; after we do this we turn to the similar calculation needed for the variance.


Letting $\psi(g) = \exp \left( - n K \lambda_1^{- g}  \right)$ and taking a first-order approximation, we see that this error term is
\be \frac{B_2}{2!}(\psi'(h_n)-\psi'(\ell_n))+R,\ee
where $R$ is less than $\frac{2\zeta(2)}{(2\pi)^2}(\psi'(h_n)-\psi'(\ell_n))$, or $\frac{1}{12}(\psi'(h_n)-\psi'(\ell_n))$. Thus
\be \left|\mbox{Error}^1_{EM}\right|\ \le\ \left|\frac{1}{6} (\psi'(h_n)-\psi'(\ell_n))\right|\ \le\ \frac{1}{6} (\left|\psi'(h_n)\right|+\left|\psi'(\ell_n)\right|).\ee
Now $\psi'(g)= nK (\log \lambda_1) \lambda_1^{-g}\exp(-nK\lambda_1^{-g})$. Since we have $h_n > \lfloor 6\log n\rfloor$ and $\ell_n =\lfloor c\log n\rfloor$, our error term becomes
\begin{eqnarray}
&&\left|\mbox{Error}^1_{EM}\right|\\
&\ll&\frac{1}{6}nK(\log \lambda_1)\left(\left|\lambda_1^{-\lfloor C\log n\rfloor}\exp(-nK\lambda_1^{-\lfloor C\log n\rfloor})\right|+\left|\lambda_1^{-\lfloor c\log n\rfloor}\exp(-nK\lambda_1^{-\lfloor c\log n\rfloor})\right|\right).\nonumber\\
\end{eqnarray}
Since $c\log\lambda_1<1$ the second term above has an exponential evaluated at a multiple of $-n^\delta$ for some $\delta > 0$, which kills the polynomial growth in $n$. Similarly our choice of $C$ shows the first term has at least a power decay in $n$, and thus $\mbox{Error}^1_{EM}=o(1)$ as claimed.\\ \

We also need to estimate $\sum_{g= \ell_n}^{h_n}  g \exp \left( - n K \lambda_1^{-g} \right)$.
Using the Euler-Maclaurin formula, we showed this sum equals
\be\int_{\ell_n}^{h_n} t \ \exp \left( - n K \lambda_1^{-t} \right)  \ dt  \ +\  \frac {1} {2} t \  P(n,t) \bigg|_{t = \ell_n}^{h_n}  \  + \ \mbox{Error}^2_{EM}.\ee

To bound $\mbox{Error}^2_{EM}$, we let $\psi(g) = g \exp \left( - n K \lambda_1^{- g}  \right)$. Since $g$ is on the order of $\log n$ in this interval, we can mimic our previous analysis, as that gave us a power savings in $n$. Thus we have $\mbox{Error}^2_{EM} =1+o(1)$, which completes our analysis of the variance of the longest gap.

\ \\


\begin{thebibliography}{BBGILMT}

\bibitem[Al]{Al}
H. Alpert,  \emph{Differences of multiple Fibonacci numbers}, Integers: Electronic Journal of Combinatorial Number Theory  \textbf{9} (2009), 745--749.

\bibitem[BBGILMT]{BBGILMT}
O. Beckwith, A. Bower, L. Gaudet, R. Insoft, S. Li, S. J. Miller and P. Tosteson, \emph{The Average Gap Distribution for Generalized Zeckendorf Decompositions}, The Fibonacci Quarterly \textbf{51} (2013), 13--27.

\bibitem[BCCSW]{BCCSW}
E. Burger, D. C. Clyde, C. H. Colbert, G. H. Shin and Z. Wang, \emph{A Generalization of a Theorem of Lekkerkerker to Ostrowski's Decomposition of Natural Numbers}, Acta Arith. \textbf{153} (2012), 217--249.

\bibitem[Day]{Day}
D. E. Daykin,  \emph{Representation of Natural Numbers as Sums of Generalized Fibonacci Numbers}, J. London Mathematical Society \textbf{35} (1960), 143--160.


\bibitem[DDKMMV]{DDKMMV}
P. Demontigny, T. Do, A. Kulkarni, S. J. Miller, D. Moon and U. Varma, \emph{Generalizing Zeckendorf's Theorem to $f$-decompositions}, preprint. \burl{http://arxiv.org/abs/1309.5599}.

\bibitem[DDKMV]{DDKMV}
P. Demontigny, T. Do, A. Kulkarni, S. J. Miller and U. Varma, \emph{A Generalization of Fibonacci Far-Difference Representations and Gaussian Behavior}, to appear in the Fibonacci Quarterly. \burl{http://arxiv.org/abs/1309.5600}.




\bibitem[DG]{DG}
M. Drmota and J. Gajdosik, \emph{The distribution of the sum-of-digits function}, J. Th\'eor. Nombr\'es Bordeaux \textbf{10} (1998), no. 1, 17--32.

\bibitem[EK]{EK}
P. Erd\H{o}s and M. Kac,  \emph{The Gaussian Law of Errors in the Theory of Additive Number Theoretic Functions}, American Journal of Mathematics \textbf{62} (1940), no. 1/4, pages 738--742.

\bibitem[FGNPT]{FGNPT}
P. Filipponi, P. J. Grabner, I. Nemes, A. Peth\"o and R. F. Tichy, \emph{Corrigendum to: ``Generalized Zeckendorf expansions''}, Appl. Math. Lett., \textbf{7} (1994), no. 6, 25--26.

\bibitem[FG]{FG}
B. E. Fristedt and L. F. Gray, \emph{A modern approach to probability theory}, Birkh\"auser, Boston, 1996.


\bibitem[Go]{Go}
S. Goldberg, \emph{Introduction to Difference Equations}, John Wiley \& Sons, 1961.

\bibitem[GT]{GT}
P. J. Grabner and R. F. Tichy, \emph{Contributions to digit expansions with respect to linear recurrences}, J. Number Theory \textbf{36} (1990), no. 2, 160--169.

\bibitem[GTNP]{GTNP}
P. J. Grabner, R. F. Tichy, I. Nemes and A. Peth\"o, \emph{Generalized Zeckendorf expansions},Appl. Math. Lett. \textbf{7} (1994), no. 2, 25--28.

\bibitem[GR]{GR}
I. S. Gradshteyn and I. M. Ryzhik, \emph{Table of integrals, series, and products} (seventh edition), Academic Press, San Diego, CA, 2007.

\bibitem[Ha]{Ha} N. Hamlin,  \emph{Representing Positive Integers as a Sum of Linear Recurrence Sequences}, Abstracts of Talks, Fourteenth International Conference on Fibonacci Numbers and Their Applications (2010), pages 2--3.


\bibitem[Ho]{Ho} V. E. Hoggatt,  \emph{Generalized Zeckendorf theorem}, Fibonacci Quarterly \textbf{10} (1972), no. 1 (special issue on representations), pages 89--93.

\bibitem[Ke]{Ke} T. J. Keller,  \emph{Generalizations of Zeckendorf's theorem}, Fibonacci Quarterly \textbf{10} (1972), no. 1 (special issue on representations), pages 95--102.

\bibitem[LT]{LT}
M. Lamberger and J. M. Thuswaldner, \emph{Distribution properties of digital expansions arising from linear recurrences}, Math. Slovaca \textbf{53} (2003), no. 1, 1--20.

\bibitem[Len]{Len} T. Lengyel, \emph{A Counting Based Proof of the Generalized Zeckendorf's Theorem}, Fibonacci Quarterly \textbf{44} (2006), no. 4, 324--325.

\bibitem[Lek]{Lek} C. G. Lekkerkerker,  \emph{Voorstelling van natuurlyke getallen door een som van getallen van Fibonacci}, Simon Stevin \textbf{29} (1951-1952), 190--195.

\bibitem[KKMW]{KKMW} M. Kolo$\breve{{\rm g}}$lu, G. Kopp, S. J. Miller and Y. Wang,  \emph{On the number of summands in Zeckendorf decompositions}, Fibonacci Quarterly \textbf{49} (2011), no. 2, 116--130.



\bibitem[MW1]{MW1} S. J. Miller and Y. Wang,  \emph{From Fibonacci Numbers to Central Limit Type Theorems}, Journal of Combinatorial Theory, Series A \textbf{119} (2012), no. 7, 1398--1413.

\bibitem[MW2]{MW2} S. J. Miller and Y. Wang,  \emph{Gaussian Behavior in Generalized Zeckendorf Decompositions},  to appear in the conference proceedings of the 2011 Combinatorial and Additive Number Theory Conference. \burl{http://arxiv.org/abs/1107.2718}.

\bibitem[Sch]{Sch}
M. Schilling, \emph{The longest run of heads}, College Math. J. \textbf{21} (1990), no. 3, 196--207.

\bibitem[Ste1]{Ste1}
W. Steiner, \emph{Parry expansions of polynomial sequences}, Integers \textbf{2} (2002), Paper A14.

\bibitem[Ste2]{Ste2}
W. Steiner, \emph{The Joint Distribution of Greedy and Lazy Fibonacci Expansions}, Fibonacci Quarterly \textbf{43} (2005), 60--69.


\bibitem[Ze]{Ze} E. Zeckendorf, \emph{Repr\'esentation des nombres naturels par une somme des nombres de Fibonacci ou de nombres de Lucas}, Bulletin de la Soci\'et\'e Royale des Sciences de Li\`ege \textbf{41} (1972), pages 179--182.



\end{thebibliography}
\end{document}